\documentclass[12pt,a4paper]{amsart}
\usepackage{amsmath,amsfonts,amssymb,amsthm}
\usepackage{dsfont,mathrsfs}
\usepackage{curves,epic,eepic}

\newtheorem{thm}{Theorem}[section]
\newtheorem{lemma}[thm]{Lemma}
\newtheorem{prop}[thm]{Proposition}
\newtheorem{cor}[thm]{Corollary}
\theoremstyle{definition}
\newtheorem{defi}[thm]{Definition}
\newtheorem{rem}[thm]{Remark}
\newtheorem{ex}[thm]{Example}
\newtheorem{con}[thm]{Construction}

\def\Z{\mathds Z}
\def\Q{\mathds Q}
\def\C{\mathds C}
\def\phi{\varphi}
\def\id{{\it id}}
\def\A{\mathds A}
\def\P{\mathds P}
\def\G{\mathds G}
\def\O{\mathcal O}

\DeclareMathOperator{\Spec}{Spec}
\DeclareMathOperator{\rk}{rk}
\DeclareMathOperator{\Desc}{Desc}

\begin{document}

\title[Generalisations of Losev-Manin moduli spaces]
{On generalisations of Losev-Manin moduli spaces
for classical root systems}

\author[Victor Batyrev]{Victor Batyrev}
\address{Mathematisches Institut, Universit\"at T\"ubingen,\newline
Auf der Morgenstelle 10, 72076 T\"ubingen, Germany}
\email{batyrev@everest.mathematik.uni-tuebingen.de}

\author[Mark Blume]{Mark Blume}
\thanks{The second author was supported 
by DFG-Schwerpunkt 1388 Darstellungstheorie.}
\address{Mathematisches Institut, Universit\"at M\"unster,\newline
Einsteinstrasse 62, 48149 M\"unster, Germany}
\email{mark.blume@uni-muenster.de}

\dedicatory{In memory of Eckart Viehweg}

\begin{abstract}
Losev and Manin introduced fine moduli spaces $\overline{L}_n$ 
of stable $n$-pointed chains of projective lines. The moduli 
space $\overline{L}_{n+1}$ is isomorphic to the toric variety 
$X(A_n)$ associated with the root system $A_n$, which is part of 
a general construction to associate with a root system $R$ of 
rank $n$ an $n$-dimensional smooth projective toric variety $X(R)$. 
In this paper we investigate generalisations of the Losev-Manin
moduli spaces for the other families of classical root systems.
\end{abstract}

\vspace*{-6mm}

\maketitle

\thispagestyle{empty}

\enlargethispage{5mm}

\section*{Introduction}

In \cite{LM00} Losev and Manin introduced fine moduli spaces
$\overline{L}_n$ of stable $n$-pointed chains of projective lines.
These Losev-Manin moduli spaces are similar to the moduli spaces
$\overline{M}_{0,n+2}$, but whereas $\overline{M}_{0,n+2}$
parametrises trees of projective lines with $n+2$ marked points
that are not allowed to coincide, the moduli space $\overline{L}_n$
parametrises chains of projective lines with two poles and
$n$ marked points that may coincide.

The Losev-Manin moduli space $\overline{L}_{n+1}$ has the structure
of an $n$-dimensional smooth projective toric variety such that the
boundary divisors parametrising reducible curves correspond to the
torus invariant divisors; it coincides with the toric variety $X(A_n)$
associated with the root system $A_n$. This is part of a general
construction to associate with a root system $R$ of rank $n$ an
$n$-dimensional smooth projective toric variety $X(R)$ (\cite{Kl85},
\cite{Pr90}). In the introduction to \cite{LM00} the authors asked
about generalisations of the moduli spaces $\overline{L}_n$ for the
other families of classical root systems. In the present paper
we address this problem.

Concerning the family of root systems of type $B$ we present a
variant of the Losev-Manin moduli problem by considering chains
of projective lines of odd length with an involution permuting
the two poles having one marked point $s_0$ invariant under
the involution and $n$ pairs of marked points $s_i^\pm$ that
are interchanged by the involution.
We show that these pointed curves admit a fine moduli space
$\overline{L}_n^{0,\pm}$ which is isomorphic to the toric variety
$X(B_n)$ such that the boundary divisors of the moduli space get
identified with the torus invariant divisors.

It is well known that for the Losev-Manin moduli spaces,
as for the moduli spaces $\overline{M}_{0,n}$, the universal
curve over $\overline{L}_{n+1}$ is the next moduli space
$\overline{L}_{n+2}$ together with a natural forgetful morphism
$\overline{L}_{n+2}\to\overline{L}_{n+1}$.
In \cite{BB11} we developed functorial properties of the
toric varieties $X(R)$ with respect to maps of root systems
and observed that this morphism $\overline{L}_{n+2}\!=\!X(A_{n+1})
\,\to\,\overline{L}_{n+1}\!=\!X(A_{n})$ is induced by the inclusion
of root systems $A_{n}\to A_{n+1}$. Furthermore, the $n\!+\!1$ sections
$X(A_{n})\to X(A_{n+1})$ come from  projections of root systems
$A_{n+1}\to A_{n}$ along the $n+1$ additional pairs of opposite
roots in $A_{n+1}$ not contained in $A_{n}$.

\pagebreak

All this generalises to the family of root systems of type $B$:
the morphism $X(B_{n+1})\to X(B_n)$ coming from the inclusion
of root systems $B_n\to B_{n+1}$ is flat and its fibres have
the structure of chains of projective lines of odd length.
The $2n+1$ additional pairs of opposite roots in $B_{n+1}$ give
$2n+1$ sections. There is a symmetry of $B_{n+1}$ fixing $B_n$
which induces an involution $I$ of $X(B_{n+1})$ over $X(B_n)$
such that the sections are grouped into $n$ pairs of sections
$s_i^\pm$ interchanged by the involution and one section $s_0$
invariant under the involution. We show that $X(B_{n+1})\to X(B_n)$
together with these sections and the involution $I$ forms the
universal family over the fine moduli space
$\overline{L}_n^{0,\pm}=X(B_n)$.
On the other hand, we will see that the toric varieties $X(R_n)$
for $R=C,D$ do not form fine moduli spaces of pointed reduced
curves having $X(R_{n+1})\to X(R_n)$ as universal family.

In the case of root systems of type $C$ the morphism
$X(C_{n+1})\to X(C_n)$ is flat with one-dimensional fibres
having the structure of $2n$-pointed chains of projective lines
of odd length with involution except that over a certain torus
invariant divisor nonreduced components occur.
On the one hand we can consider families of pointed curves as in
the $B_n$-case but without the section $s_0$ and thereby allowing
an additional involution as isomorphism. This gives rise to
a toric Deligne-Mumford stack $\mathcal X(C_n)$ which is an
orbifold having the toric variety $X(C_n)$ as coarse moduli
space with stacky points over the divisor determined by the
nonreduced fibres.
On the other hand we can describe $X(C_n)$ as a fine moduli
space $\overline{L}_n^\pm$ of $2n$-pointed chains of projective lines
of odd and even length with involution with each of the marked points
corresponding to a pair of opposite roots in $C_{n+1}\setminus C_n$
that defines a projection $C_{n+1}\to C_n$.
The universal family arises from $X(C_{n+1})\to X(C_n)$ by
contracting the nonreduced components in the fibres.

In the case of the remaining family of root systems of type $D$
the morphism $X(D_{n+1})\to X(D_n)$ is not flat. There are
$2$-dimensional fibres that occur over closures of certain torus
orbits of codimension $2$, over the other points as fibres we have
$2n$-pointed chains of projective lines with involution.

We observe that in the cases of all families of root systems
$R=A,B,C,D$ the torus fixed points of $X(R_n)$ correspond to
pointed curves having the form of the Dynkin diagram for the
root system $R_{n+1}$.

\medskip\enlargethispage{13mm}

{\bf Outline of the paper.} In the first sections 1--5 we deal
with the case of root systems of type $B$.
In section \ref{sec:moduliprobl} we formulate a moduli problem
of $(2n+1)$-pointed chains of projective lines called $B_n$-curves,
which is a variant of the Losev-Manin moduli problem.
In section \ref{sec:X(Bn)} we collect some facts about the toric
varieties $X(B_n)$ associated with root systems of type $B$.
Section \ref{sec:univcurve} is about the morphism
$X(B_{n+1})\to X(B_n)$, which, together with its sections and the
involution, forms a flat family of $B_n$-curves, and in section
\ref{sec:modspace} we prove that the toric variety
$\overline{L}_n^{0,\pm}\!=\!X(B_n)$ is a fine moduli space of $B_n$-curves
with universal family $X(B_{n+1})\to X(B_n)$.
To show that the moduli functor of $B_n$-curves is isomorphic
to the functor of the toric variety $X(B_n)$ we use the description
of the functor of toric varieties associated with root systems
given in \cite[1.3]{BB11}; our proof is a variation of our new
proof of the respective statement for root systems of type $A$
given in \cite[3.3]{BB11}.
In section \ref{sec:cohomBn} we present some results on the
(co)homology of the spaces $\overline{L}_n^{0,\pm}\!=\!X(B_n)$,
giving descriptions similar to the case of the Losev-Manin moduli
spaces $\overline{L}_{n+1}\!=\!X(A_n)$.

In the remaining sections \ref{sec:C_n} and \ref{sec:D_n} the cases
of the root systems of type $C$ and $D$ are investigated.

\pagebreak
\section{Pointed chains of projective lines with involution}
\label{sec:moduliprobl}

\begin{defi}
A {\sl chain of projective lines} of length $m$ over an algebraically
closed field $K$ is a projective curve $C=C_1\cup\ldots\cup C_m$
over $K$ such that every irreducible component $C_j$ of $C$ is a
projective line with poles $p^-_j,p^+_j$ and these components intersect
as follows: different components $C_i$ and $C_j$ intersect only if
$|i-j|=1$ and in this case $C_j,C_{j+1}$ intersect transversally in
the single point $p^+_j=p^-_{j+1}$.
For $p^-_1\in C_1$ and $p^+_m\in C_m$ we write $s_-$ and $s_+$.
Two chains of projective lines $(C,s_-,s_+)$ and $(C',s_-',s_+')$
are called isomorphic if there is an isomorphism $\phi\colon C\to C'$
such that $\phi(s_-)=s_-'$, $\phi(s_+)=s_+'$.
\end{defi}

\begin{defi}\label{def:chaininv}
A {\sl chain of projective lines with involution} $(C,I,s_-,s_+)$ is
a chain of projective lines together with an isomorphism $I\colon C\to C$
such that $I^2=\id_C$ and $I(s_-)=s_+$.
In this case we use the following notation: if the chain has odd length
denote by $(C_0,p^-_0,p^+_0)$ the central component; denote by
$(C_j,p^-_j,p^+_j)$, $(C_{-j},p^-_{-j},p^+_{-j})$
the pairs of $I$-conjugate components (i.e.\ $I(C_j)=C_{-j}$,
$I(p^-_j)=p^+_{-j}$, $I(p^+_j)=p^-_{-j}$) such that
$p^+_j=p^-_{j+1}$, $p^-_{-j}=p^+_{-(j+1)}$ and in case of odd lenght
$p^+_0=p^-_1$, $p^-_0=p^+_{-1}$ whereas in case of even length
$p^+_{-1}=p^-_1$. In particular, we have $s_-=p^-_{-m}$, $s_+=p^+_m$
if the  chain has length $2m$ or $2m+1$.
Two chains of projective lines with involution $(C,I,s_-,s_+)$ and
$(C',I',s_-',$ $s_+')$ are called isomorphic if there is an isomorphism
of chains of projective lines $\phi\colon(C,s_-,s_+)\to(C',s_-',s_+')$
such that $\phi\circ I=I'\circ\phi$.
\end{defi}

In the following we are concerned with certain compactifications
of the algebraic torus $(2\G_m)^n$ parametrising $n$ pairs
of points of the form $(z,\frac{1}{z})$ in $(\G_m,1)\subset
(\P^1,0,\infty,1)$, i.e.\ pairs of points which are interchanged
by the involution of $\P^1$ that fixes the point $1$ and interchanges
the two poles $0$ and $\infty$.
These compactifications, which will be associated with root systems,
parametrise isomorphism classes of certain pointed chains of projective
lines with an involution.
We now define the type of pointed curve which will be relevant
in the case of root systems of type $B$.

\begin{defi}
A {\sl $(2n\!+\!1)$-pointed chain of projective lines with involution}
$(C,I,s_-,$ $s_+,s_0,s_1^\pm,\ldots,s_n^\pm)$ is a chain of
projective lines with involution $(C,I,s_-,s_+)$ of odd length
together with (possibly coinciding) marked points $s_0,s_i^\pm\in C$
different from the poles such that $I(s_0)=s_0$, $I(s_i^-)=s_i^+$. 
Two $(2n\!+\!1)$\,-\,pointed chains of projective lines with involution
$(C,I,s_-,s_+,s_0, s_1^\pm,\ldots,s_n^\pm)$ and
$(C',I',s_-',s_+',s_0',{s_1'}^\pm,\ldots,{s_n'}^\pm)$ are called
isomorphic if there is an isomorphism $\phi\colon
(C,I,s_-,s_+)\to(C',I',s_-',s_+')$ of the underlying chains of
projective lines with involution such that $\phi(s_0)=s_0'$,\linebreak
$\phi(s_j^\pm)={s_j'}^\pm$. A $(2n+1)$-pointed chain of projective
lines with involution 
$(C,I,s_-,s_+,\linebreak s_0,s_1^\pm,\ldots,s_n^\pm)$ 
is called {\sl stable} if each component of $C$ contains at least 
one of the points $s_0,s_j^\pm$. A {\sl $B_n$-curve} over an
algebraically closed field $K$ is a stable $(2n+1)$-pointed chain of
projective lines over $K$.
\end{defi}

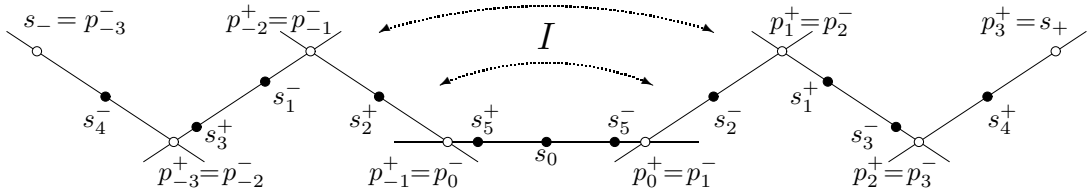
\medskip
\noindent
\begin{picture}(150,25)(0,0)

\put(75,6){\line(-1,0){20}}\put(75,6){\line(1,0){20}}

\put(97,12){\line(-3,-2){13}}\put(97,12){\line(3,2){13}}
\put(115,12){\line(-3,2){13}}\put(115,12){\line(3,-2){13}}
\put(133,12){\line(-3,-2){13}}\put(133,12){\line(3,2){13}}

\put(53,12){\line(-3,2){13}}\put(53,12){\line(3,-2){13}}
\put(35,12){\line(-3,-2){13}}\put(35,12){\line(3,2){13}}
\put(17,12){\line(-3,2){13}}\put(17,12){\line(3,-2){13}}

\filltype{white}
\put(8,18){\circle*{1.2}}
\put(13,20){\makebox(0,0)[b]{\Small$s_-\!=p_{-3}^-$}}
\put(26,6){\circle*{1.2}}
\put(31,4){\makebox(0,0)[t]{\Small$p_{-3}^+\!\!=\!p_{-2}^-$}}
\put(44,18){\circle*{1.2}}
\put(40,20){\makebox(0,0)[b]{\Small$p_{-2}^+\!\!=\!p_{-1}^-$}}
\put(62,6){\circle*{1.2}}
\put(58,4){\makebox(0,0)[t]{\Small$p_{-1}^+\!\!=\!p_{0}^-$}}

\put(142,18){\circle*{1.2}}
\put(138,20){\makebox(0,0)[b]{\Small$p_3^+\!=\!s_+$}}
\put(124,6){\circle*{1.2}}
\put(121,4){\makebox(0,0)[t]{\Small$p_{2}^+\!\!=\!p_{3}^-$}}
\put(106,18){\circle*{1.2}}
\put(110,20){\makebox(0,0)[b]{\Small$p_{1}^+\!\!=\!p_{2}^-$}}
\put(88,6){\circle*{1.2}}
\put(92,4){\makebox(0,0)[t]{\Small$p_{0}^+\!\!=\!p_{1}^-$}}

\filltype{black}
\put(75,6){\circle*{1.2}}\put(75,5){\makebox(0,0)[t]{\Small$s_0$}}
\put(66,6){\circle*{1.2}}\put(67,7){\makebox(0,0)[b]{\Small$s_5^+$}}
\put(84,6){\circle*{1.2}}\put(85,7){\makebox(0,0)[b]{\Small$s_5^-$}}
\put(17,12){\circle*{1.2}}\put(16,11){\makebox(0,0)[t]{\Small$s_4^-$}}
\put(133,12){\circle*{1.2}}\put(135,11){\makebox(0,0)[t]{\Small$s_4^+$}}
\put(53,12){\circle*{1.2}}\put(51,11){\makebox(0,0)[t]{\Small$s_2^+$}}
\put(97,12){\circle*{1.2}}\put(99,11){\makebox(0,0)[t]{\Small$s_2^-$}}
\put(29,8){\circle*{1.2}}\put(32,9){\makebox(0,0)[t]{\Small$s_3^+$}}
\put(121,8){\circle*{1.2}}\put(117,9){\makebox(0,0)[t]{\Small$s_3^-$}}
\put(38,14){\circle*{1.2}}\put(41,14){\makebox(0,0)[t]{\Small$s_1^-$}}
\put(112,14){\circle*{1.2}}\put(109,14){\makebox(0,0)[t]{\Small$s_1^+$}}

\put(75,20){\makebox(0,0)[c]{\large$I$}}
\put(62,14){\vector(-2,-1){1}}\put(88,14){\vector(2,-1){1}}
\qbezier[50](62,14)(75,19)(88,14)

\put(54,21){\vector(-2,-1){1}}\put(96,21){\vector(2,-1){1}}
\qbezier[80](54,21)(75,27)(96,21)
\end{picture}

\medskip

\begin{defi}
Let $Y$ be a scheme. A {\sl $B_n$-curve over $Y$} is a collection
$(\pi\colon C\to Y,I,$ $s_-,s_+,s_0,s_1^\pm,\ldots,s_n^\pm)$, where $C$
is a scheme, $\pi$ is a flat proper morphism of schemes, $I\colon
C\to C$ an involution over $Y$ and
$s_-,s_+,s_0,s_1^\pm,\ldots,s_n^\pm\colon Y\to C$ are sections such
that for any geometric point $y$ of $Y$ the collection
$(C_y,I_y,(s_-)_y,(s_+)_y,(s_0)_y,(s_1^\pm)_y,$
$\ldots,(s_n^\pm)_y)$ is a $B_n$-curve over $y$. An isomorphism of
$B_n$-curves over $Y$ is an isomorphism of $Y$-schemes that is
compatible with the involution and the sections. We define the {\sl
moduli functor of $B_n$-curves} as the functor
\[
\begin{array}{crcl}
\underline{\overline{L}_n^{0,\pm}}:\;&(\textit{schemes})^\circ&
\;\to\;&(\textit{sets})\\
&Y&\;\mapsto\;&\left\{\textit{$B_n$-curves over $Y$}\right\}\,/\sim
\end{array}
\]
that associates to a scheme $Y$ the set of isomorphism classes
of $B_n$-curves over $Y$ and to a morphism of schemes
the map obtained by pulling back $B_n$-curves.
\end{defi}

We will show in section \ref{sec:modspace} that a fine moduli
space of $B_n$-curves $\overline{L}_n^{0,\pm}$ exists and that it
is isomorphic to the toric variety associated with the root
system $B_n$.

\medskip
\section{Toric varieties $X(B_n)$}
\label{sec:X(Bn)}

For a root system $R$ of rank $n$ we have the $n$-dimensional
smooth projective toric variety $X(R)$ associated with the fan
that consists of the Weyl chambers of the root system and their
faces (\cite{Kl85}, \cite{Pr90}, see also \cite[1.1]{BB11}).
Here we consider the particular case of root systems of type $B$.

\medskip

Let $E$ be an $n$-dimensional Euclidean space with basis $u_1,\ldots,u_n$.
The root system $B_n$ in $E$ consists of the following $2n^2$ roots:
\[\pm u_i\;\:\textit{for}\;\:i\in\{1,\ldots,n\};\quad\pm(u_i+u_j),
\pm(u_i-u_j)\;\:\textit{for}\;\:i,j\in\{1,\ldots,n\},i<j.\]
The root lattice $M(B_n)\cong\Z^n$ of the root system $B_n$ is the
lattice in $E$ generated by $u_1,\ldots,u_n$. The following is a
set of simple roots:
\[u_1-u_2,u_2-u_3,\ldots,u_{n-1}-u_n,u_n.\]
The Weyl group $(\Z/2\Z)^n\rtimes S_n$ acts by $u_i\mapsto\pm u_i$
and by permuting the $u_i$.
So there are $2^nn!$ sets of simple roots, these are of the form
$\varepsilon_1u_{i_1}-\varepsilon_2u_{i_2},
\varepsilon_2u_{i_2}-\varepsilon_3u_{i_3},\ldots,
\varepsilon_{n-1}u_{i_{n-1}}-\varepsilon_{n}u_{i_{n}},
\varepsilon_nu_{i_n}$ for orderings $i_1,\ldots,i_n$ of the set
$\{1,\ldots,n\}$ and signs $\varepsilon_1,\ldots,\varepsilon_n$.
For later use we list linear relations between positive roots
of $B_n$.

\begin{lemma}\label{le:linrel-B_n}
Let $B_n^+$ be the set of positive roots of $B_n$ corresponding to the
set of simple roots $u_1-u_2,u_2-u_3,\ldots,u_{n-1}-u_n,u_n$ and put
$\beta_{ij}=u_i-u_j$, $\gamma_{ij}=u_i+u_j$ for $i,j\in\{1,\ldots,n\}$,
$i\neq j$.
Then $B_n^+=\{u_1,\ldots,u_n\}\cup\{\beta_{ij}\:|\:i<j\}\cup
\{\gamma_{ij}\:|\:i\neq j\}$ and the tripels of positive roots
$\alpha,\beta,\gamma\in B_n^+$ satisfying $\alpha+\beta=\gamma$
are the following:
\[
\begin{array}{cl}
\beta_{ij}+u_j=u_i&\quad (i,j\in\{1,\ldots,n\},\:i<j)\\
u_i+u_j=\gamma_{ij}&\quad (i,j\in\{1,\ldots,n\},\:i\neq j)\\
\beta_{ij}+\beta_{jk}=\beta_{ik}&\quad (i,j,k\in\{1,\ldots,n\},\:i<j<k)\\
\beta_{ij}+\gamma_{jk}=\gamma_{ik}&\quad (i,j,k\in\{1,\ldots,n\},\:i<j,\:
k\neq i,j)\\
\end{array}
\]
\end{lemma}

\medskip

Let $N(B_n)$ be the lattice dual to the root lattice $M(B_n)$ and
$v_1,\ldots,v_n$ the basis of $N(B_n)$ dual to $u_1,\ldots,u_n$.
The fan $\Sigma(B_n)$ is defined as the fan of Weyl chambers
in $N(B_n)$, i.e.\ its maximal cones are the Weyl chambers
$\sigma_S=S^\vee=\{v\in N(B_n)_\Q\:|\:\forall\:\alpha\in S\colon\:
\langle\alpha,v\rangle\geq 0\}$ for sets of simple roots $S$ of
the root system $B_n$ and all cones arise as faces of these.
For the set of simple roots $S=\{u_1-u_2,u_2-u_3,\ldots,u_{n-1}-u_n,u_n\}$
has the dual basis $v_1,v_1+v_2,\ldots,v_1+\ldots+v_n$ of $N(B_n)$,
the Weyl chamber $\sigma_S$ is equal to $\langle v_1,v_1+v_2,\ldots,
v_1+\ldots+v_n\rangle_{\Q_{\geq0}}$.
All Weyl chambers, i.e.\ all maximal cones of the fan $\Sigma(B_n)$,
arise as translates of $\sigma_S$ under the action of the Weyl group
on $N(B_n)_\Q$, thus they are generated by sets of elements of the form
$\varepsilon_1v_{i_1},\varepsilon_1v_{i_1}+\varepsilon_2v_{i_2},\ldots,
\varepsilon_1v_{i_1}+\ldots+\varepsilon_nv_{i_n}$ for orderings
$i_1,\ldots,i_n$ of the set $\{1,\ldots,n\}$ and signs
$\varepsilon_i\in\{\pm1\}$.
The fan $\Sigma(B_n)$ has $3^n-1$ one-dimensional cones generated by
the elements of the form $\varepsilon_1v_{i_1}+\ldots
+\varepsilon_kv_{i_k}$ for $k\in\{1,\ldots,n\}$.
These are via $v_B:=\sum_{\varepsilon_ii\in B}\varepsilon_iv_i$
$\leftrightarrow$ $B$ in bijection with the set $\mathcal B$ of all
subsets $\emptyset\neq B\subset\{\pm 1,\ldots,\pm n\}$ such that
$B\cap\{i,-i\}\neq\{i,-i\}$ for $i=1,\ldots,n$.
The one-dimensional cones for a family of such sets
$B^{(1)},\ldots,B^{(k)}$ form a higher dimensional cone whenever
they can be ordered such that $B^{(i_1)}\subsetneq\ldots\subsetneq
B^{(i_k)}$.

\medskip

We have the $n$-dimensional smooth projective toric variety $X(B_n)$
associated with the fan $\Sigma(B_n)$ with respect to the lattice
$N(B_n)$.
As usual, any element $u\in M(B_n)$ defines a character of the open
dense torus $T(B_n)$ (resp.\ a rational function on $X(B_n)$) denoted
by $x^u$.
The toric variety $X(B_n)$ has the following covering by affine spaces.
For any set $S$ of simple roots we have the maximal cone
$\sigma_S=S^\vee$ and the chart $U_S=\Spec\Z[\sigma_S^\vee\cap M(B_n)]
\cong\A^n$, for example if $S=\{u_1-u_2,u_2-u_3,\ldots,u_{n-1}-u_n,u_n\}$
then $\Z[\sigma_S^\vee\cap M(B_n)]=\Z[\frac{x_1}{x_2},\ldots,
\frac{x_{n-1}}{x_n},x_n]$. The Weyl group $W(B_n)=(\Z/2\Z)^n\rtimes S_n$
acts on $X(B_n)$, it permutes these affine charts.

\medskip

By \cite[1.2]{BB11} the closures of torus orbits in $X(B_n)$ are
isomorphic to products
$X(B_{n_1})\times X(A_{n_2})\times\ldots\times X(A_{n_k})$.
The torus invariant divisor for the one-dimensional cone generated
by $\varepsilon_1v_{i_1}+\ldots+\varepsilon_kv_{i_k}$ is isomorphic
to $X(B_{n-k})\times X(A_{k-1})$, in particular for
$\varepsilon_1v_{i_1}+\ldots+\varepsilon_nv_{i_n}$ we have a divisor
isomorphic to $X(A_{n-1})$.

\medskip
\section{The universal curve}
\label{sec:univcurve}

We construct a $B_n$-curve over $X(B_n)$, which later turns out to
be the universal curve over the moduli space of $B_n$-curves,
by using functorial properties of the toric varieties associated
with root systems developed in \cite[1.2]{BB11}. We fix the
following notations for roots of $B_n$ and $B_{n+1}$:
$\beta_{ij}=u_i-u_j$, $\gamma_{ij}=u_i+u_j$ for $i,j\in\{1,\ldots,n\}$,
$i\neq j$ and $\alpha_i^+=u_{n+1}+u_i$, $\alpha_i^-=u_{n+1}-u_i$
for $i\in\{1,\ldots,n\}$.

\begin{con}\label{con:univ-B_n-curve} (The universal $B_n$-curve.)\\
Consider the root subsystem $B_n\subset B_{n+1}$ consisting of the
roots in the subspace spanned by $u_1,\ldots,u_n$. The inclusion of
root systems $B_n\subset B_{n+1}$ determines a proper surjective
morphism $X(B_{n+1})\to X(B_n)$.

There are $2n+1$ additional pairs of opposite roots, the pairs
$\pm\alpha_i^+$ and $\pm\alpha_i^-$ for $i\in\{1,\ldots,n\}$ and
the pair $\pm u_{n+1}$. Each of these defines a projection onto
the root subsystem $B_n\subset B_{n+1}$ in the sense of \cite[1.2]{BB11},
thus the pairs $\pm\alpha_i^+$ and $\pm\alpha_i^-$ define
sections $s_i^+,s_i^-\colon X(B_n)\to X(B_{n+1})$
and an additional section $s_0\colon X(B_n)\to X(B_{n+1})$ is
given by the projection with kernel generated by $u_{n+1}$.

Further, we have two sections $s_\pm\colon X(B_n)\to X(B_{n+1})$
which are inclusions of $X(B_n)$ into $X(B_{n+1})$ as torus invariant
divisors (cf.\ \cite[Prop.\ 1.9, Rem.\ 1.11]{BB11}) corresponding
to the one-dimensional cones of the fan $\Sigma(B_{n+1})$ generated
by $\pm v_{n+1}$.
Locally, the image of $s_-$ (resp.\ $s_+$) is given by the equations
$x^{-\alpha_i^\pm}=0$, $x^{-u_{n+1}}=0$ (resp.\ $x^{\alpha_i^\pm}=0$,
$x^{u_{n+1}}=0$) on the affine charts of $X(B_{n+1})$ corresponding
to the sets of positive roots containing $-\alpha_i^\pm,-u_{n+1}$
(resp.\ $\alpha_i^\pm, u_{n+1}$).

There is an involution $I$ of $X(B_{n+1})$ over $X(B_n)$ corresponding
to the involution of $B_{n+1}$ which fixes $B_n\subset B_{n+1}$
determined by the linear map $u_i\mapsto u_i$ for $i\in\{1,\ldots,n\}$
and $u_{n+1}\mapsto-u_{n+1}$. This element of the Weyl group
$W(B_{n+1})$ is the reflection determined by the root $\pm u_{n+1}$.
The section $s_0$ is invariant under $I$, whereas for each
$i\in\{1,\ldots,n\}$ the sections $s_i^+$ and $s_i^-$ and also
$s_-$ and $s_+$ are interchanged.
\end{con}

\begin{prop}
The collection $(X(B_{n+1})\to X(B_n),I,s_-,s_+,s_0,s_1^\pm,\ldots,
s_n^\pm)$ of construction \ref{con:univ-B_n-curve} is a $B_n$-curve
over $X(B_n)$.
\end{prop}
\begin{proof}
The morphism $X(B_{n+1})\to X(B_n)$ is proper. We can show that
it is flat by considering the covering of $X(B_{n+1})$ and
$X(B_n)$ by affine toric charts similar as in the case of root
systems of type $A$ (see \cite[Prop.\ 3.7]{BB11}).

That any fibre is a $B_n$-curve follows from the results below.
Remark \ref{rem:emb-univ-B_n-curve} describes the universal curve
in terms of equations, proposition \ref{prop:emb-B_n-curve} shows
that such equations define a $B_n$-curve.
It only remains to show that any $B_n$-data arises as in proposition
\ref{prop:emb-B_n-curve} from a $B_n$-curve. This is done in
lemma \ref{le:B_n-data--B_n-curve}.
\end{proof}

\begin{defi}
We call the object $(X(B_{n+1})\to X(B_n),I,s_-,s_+,s_0,s_1^\pm,\ldots,
s_n^\pm)$ of construction \ref{con:univ-B_n-curve} the universal
$B_n$-curve over $X(B_n)$.
\end{defi}

\begin{ex} We picture the universal curve $X(B_2)$ over
$X(B_1)\cong\P^1$ with the sections $s_-,s_+,s_0,s_1^\pm$.
The generic fibre is a $\P^1$, whereas the fibres over the
two torus fixed points $x^{-u}=0$ and $x^u=0$ of $X(B_1)$
are chains consisting of three projective lines.

\noindent
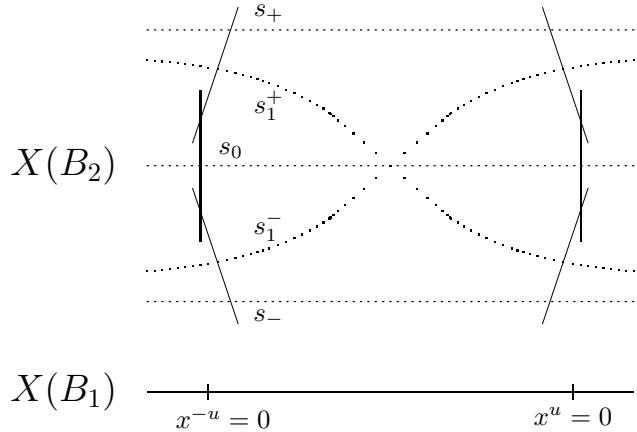
\begin{picture}(150,60)(0,0)
\put(35,-5){
\put(-10,40){\makebox(0,0)[l]{\large$X(B_2)$}}
\put(-10,10){\makebox(0,0)[l]{\large$X(B_1)$}}
\put(8,10){\line(1,0){64}}
\put(16,9){\line(0,1){2}}\put(18,8){\makebox(0,0)[t]{\Small$x^{-u}=0$}}
\put(64,9){\line(0,1){2}}\put(64,8){\makebox(0,0)[t]{\Small$x^u=0$}}
\put(17,28){\line(-1,3){3}}\put(17,28){\line(1,-3){3}}
\put(17,52){\line(1,3){3}}\put(17,52){\line(-1,-3){3}}
\put(15,40){\line(0,1){10}}\put(15,40){\line(0,-1){10}}
\put(63,28){\line(1,3){3}}\put(63,28){\line(-1,-3){3}}
\put(63,52){\line(-1,3){3}}\put(63,52){\line(1,-3){3}}
\put(65,40){\line(0,1){10}}\put(65,40){\line(0,-1){10}}
\dottedline{1}(8,22)(72,22)\dottedline{1}(8,58)(72,58)
\dottedline{1}(8,40)(72,40)
\put(24,59){\makebox(0,0)[b]{\Small$s_+$}}
\put(24,21){\makebox(0,0)[t]{\Small$s_-$}}
\put(19,41){\makebox(0,0)[b]{\Small$s_0$}}
\put(24,50){\makebox(0,0)[t]{\Small$s_1^+$}}
\put(24,30){\makebox(0,0)[b]{\Small$s_1^-$}}
\curve[8](8,26, 16,27, 24,29, 32,33,  40,40,
48,47, 56,51, 64,53, 72,54)
\curve[8](8,54, 16,53, 24,51, 32,47,  40,40,
48,33, 56,29, 64,27, 72,26)
}
\end{picture}

\medskip
\noindent
This universal curve is constructed using the root system $B_2$ with
its root subsystem $B_1=\{\pm u_1\}$ and the corresponding map of
fans $\Sigma(B_2)\to\Sigma(B_1)$.

\medskip

\noindent
\begin{picture}(150,60)(0,0)
\put(3,40){\makebox(0,0)[l]{\large$B_2$}}
\put(3,5){\makebox(0,0)[l]{\large$B_1$}}
\put(35,40){\vector(0,1){18}}\put(35,40){\vector(0,-1){18}}
\put(35,40){\vector(1,0){18}}\put(35,40){\vector(-1,0){18}}
\put(35,40){\vector(1,1){18}}\put(35,40){\vector(-1,1){18}}
\put(35,40){\vector(1,-1){18}}\put(35,40){\vector(-1,-1){18}}
\dottedline{1}(13,44)(57,44)\dottedline{1}(13,36)(57,36)
\dottedline{1}(13,44)(13,36)\dottedline{1}(57,44)(57,36)
\put(53,37){\makebox(0,0)[b]{\small$u_1$}}
\put(17,37){\makebox(0,0)[b]{\small$-u_1$}}
\put(51,54){\makebox(0,0)[l]{\small$\alpha_1^+\!=\!u_1\!+\!u_2$}}
\put(36,57){\makebox(0,0)[l]{\small$u_2$}}
\put(19,54){\makebox(0,0)[r]{\small$\alpha_1^-$}}
\put(51,26){\makebox(0,0)[l]{\small$-\alpha_1^-\!=\!u_1\!-\!u_2$}}
\put(34,23){\makebox(0,0)[r]{\small$-u_2$}}
\put(19,26){\makebox(0,0)[r]{\small$-\alpha_1^+$}}
\put(25,33){\vector(0,1){2}}\put(45,33){\vector(0,1){2}}
\dottedline{1}(25,8)(25,35)\dottedline{1}(45,8)(45,35)
\put(35,5){\vector(1,0){18}}\put(35,5){\vector(-1,0){18}}
\put(53,7){\makebox(0,0)[b]{\small$u$}}
\put(17,7){\makebox(0,0)[b]{\small$-u$}}
\drawline(35,4)(35,6)

\put(75,0){
\put(72,40){\makebox(0,0)[r]{\large$\Sigma(B_2)$}}
\put(72,5){\makebox(0,0)[r]{\large$\Sigma(B_1)$}}
\put(35,40){\vector(0,1){18}}\put(35,40){\vector(0,-1){18}}
\put(35,40){\vector(1,0){18}}\put(35,40){\vector(-1,0){18}}
\put(35,40){\vector(1,1){18}}\put(35,40){\vector(-1,1){18}}
\put(35,40){\vector(1,-1){18}}\put(35,40){\vector(-1,-1){18}}
\put(53,39){\makebox(0,0)[t]{\small$v_1$}}
\put(17,39){\makebox(0,0)[t]{\small$-v_1$}}
\put(51,54){\makebox(0,0)[l]{\small$v_1\!+\!v_2$}}
\put(36,57){\makebox(0,0)[l]{\small$v_2$}}
\put(19,54){\makebox(0,0)[r]{\small$-v_1\!+\!v_2$}}
\put(51,26){\makebox(0,0)[l]{\small$v_1\!-\!v_2$}}
\put(34,23){\makebox(0,0)[r]{\small$-v_2$}}
\put(19,26){\makebox(0,0)[r]{\small$-v_1\!-\!v_2$}}
\dottedline{1}(25,18)(25,12)\put(25,12){\vector(0,-1){2}}
\dottedline{1}(45,18)(45,12)\put(45,12){\vector(0,-1){2}}
\put(35,5){\vector(1,0){18}}\put(35,5){\vector(-1,0){18}}
\put(53,7){\makebox(0,0)[b]{\small$v$}}
\put(17,7){\makebox(0,0)[b]{\small$-v$}}
\drawline(35,4)(35,6)
}
\end{picture}
\end{ex}

By \cite[1.2]{BB11} pairs of opposite roots $\{\pm\alpha\}$ in a
root system $R$ give rise to morphisms $X(R)\to\P^1$. We write
$\P^1_{\{\pm\alpha\}}$ for the corresponding copy of $\P^1$ with
homogeneous coordinates $z_\alpha,z_{-\alpha}$ such that the
rational function $x^\alpha$ on $X(R)$ is the pull-back of
$\frac{z_\alpha}{z_{-\alpha}}$.
Further, the collection of these morphisms for all pairs of
opposite roots $\{\pm\alpha\}$ in $R$, i.e.\ root subsystems
isomorphic to $A_1$, defines a closed embedding
$X(R)\to\prod_{\{\pm\alpha\}\subseteq R}\P^1_{\{\pm\alpha\}}=:P(R)$.
By \cite[1.3]{BB11} the equations defining the image of $X(R)$
in $P(R)$ come from root subsystems of type $A_2$ in $R$ or
equivalently linear relations between positive roots of $R$.

\begin{rem}\label{rem:emb-univ-B_n-curve}
Consider $X(B_{n+1})$ and $X(B_n)$ as embedded $X(B_{n+1})\subseteq
P(B_{n+1})$, $X(B_n)\subseteq P(B_n)$.
Then the morphism $X(B_{n+1})\to X(B_n)$ is induced by
the projection onto the subproduct $P(B_{n+1})\to P(B_n)$.
The subvarieties $X(B_{n+1})$ (resp.\ $X(B_n)$) 
are determined by the homogeneous equations
$z_{\alpha}z_{\beta}z_{-\gamma}=z_{-\alpha}z_{-\beta}z_{\gamma}$
for roots $\alpha,\beta,\gamma$ such that $\alpha+\beta=\gamma$,
i.e.\ root subsystems of type $A_2$ in $B_{n+1}$ (resp.\ $B_n$).\\
If we first consider the product $P(B_{n+1})$ and the equations
coming from root subsystems of type $A_2$ in $B_n$, we have
\[
\begin{array}{rcl}
P(B_{n+1}/B_n)_{X(B_n)}&=&\big(\prod_{A_1\cong R\subseteq B_{n+1}
\setminus B_n}\P^1_{R}\big)_{X(B_n)}\\
&=&\big(\prod_{i=1}^n\P^1_{\{\pm\alpha_i^+\}}\times
\prod_{i=1}^n\P^1_{\{\pm\alpha_i^-\}}\times
\P^1_{\{\pm u_{n+1}\}}\big)_{X(B_n)}\\
\end{array}
\]
Therein, $X(B_{n+1})$ is the closed subvariety given by the equations
corresponding to root subsystems of type $A_2$ in $B_{n+1}$
which are not contained in $B_n$. We choose the set of positive 
roots $B_{n+1}^+$ corresponding to the set of simple roots
$\{u_{n+1}-u_1,u_1-u_2,u_2-u_3,\ldots,u_{n-1}-u_n,u_n\}$.
Then $B_{n+1}^+\setminus B_n^+=\{u_{n+1},\alpha_1^\pm,\ldots,
\alpha_n^\pm\}$ and we can write these equations as follows
\begin{equation}\label{eq:univ-B_n-curveI}
t_{\beta}z_{\alpha_2}z_{-\alpha_1}=t_{-\beta}z_{-\alpha_2}z_{\alpha_1}
\qquad
\begin{array}{l}
\textit{for}\;\alpha_1,\alpha_2\in \{u_{n+1},\alpha_1^\pm,
\ldots,\alpha_n^\pm\}\\
\textit{such that}\;\beta=\alpha_1-\alpha_2\;\textit{is a root of $B_n$}\\
\end{array}
\end{equation}
where $t_{\pm\beta}$ are the homogeneous coordinates of
$\P^1_{\{\pm\beta\}}$ (consider $X(B_n)$ as embedded in $P(B_n)$)
or equivalently the two generating sections of the line bundle
$\mathscr L_{\{\pm\beta\}}$ being part of the universal data on
$X(B_n)$ as defined in \cite[1.3]{BB11}.\\
The sections $s_i^\pm\colon X(B_n)\to X(B_{n+1})$ for
$i\in\{1,\ldots,n\}$ are given by the additional equations
$z_{\alpha_i^\pm}=z_{-\alpha_i^\pm}$, the section $s_0$ by
$z_{u_{n+1}}=z_{-u_{n+1}}$.
The sections $s_-$ (resp.\ $s_+$) are given by $z_{-u_{n+1}}=0$,
$z_{-\alpha_i^\pm}=0$ for $i=1,\ldots,n$ (resp.\ $z_{u_{n+1}}=0$,
$z_{\alpha_i^\pm}=0$ for $i=1,\ldots,n$).
\end{rem}

\begin{ex}
The universal $B_1$-curve $X(B_2)\!\subset\!(\P^1_{\{\pm\alpha_1^+\}}
\!\times\!\P^1_{\{\pm\alpha_1^-\}}\!\times\!\P^1_{\{\pm u_2\}})_{X(B_1)}$
over $X(B_1)$ is given by the homogeneous equations
\[
t_{u_1}z_{u_2}z_{-\alpha_1^+}=t_{-u_1}z_{-u_2}z_{\alpha_1^+},\quad
t_{u_1}z_{\alpha_1^-}z_{-u_2}=t_{-u_1}z_{-\alpha_1^-}z_{u_2}
\]
where $(\mathscr L_{\{\pm u_1\}},\{t_{u_1},t_{-u_1}\})$ is the
universal $B_1$-data on $X(B_1)\cong\P^1$. We picture the
$B_1$-curves defined by these equations for
$(t_{u_1}:t_{-u_1})=(0:1),(a:b),(1:0)$. If $(t_{u_1}:t_{-u_1})=(a:b)
\neq(0:1),(1:0)$ we have a projective line, we draw its projection
onto $\P^1_{\{\pm u_2\}}$ and write the sections in terms of
the homogeneous coordinates $z_{u_2},z_{-u_2}$.
Over the two torus fixed points of $X(B_1)$ the curve is a chain of
three projective lines in $\P^1_{\{\pm\alpha_1^+\}}\times
\P^1_{\{\pm\alpha_1^-\}}\times\P^1_{\{\pm u_2\}}$.

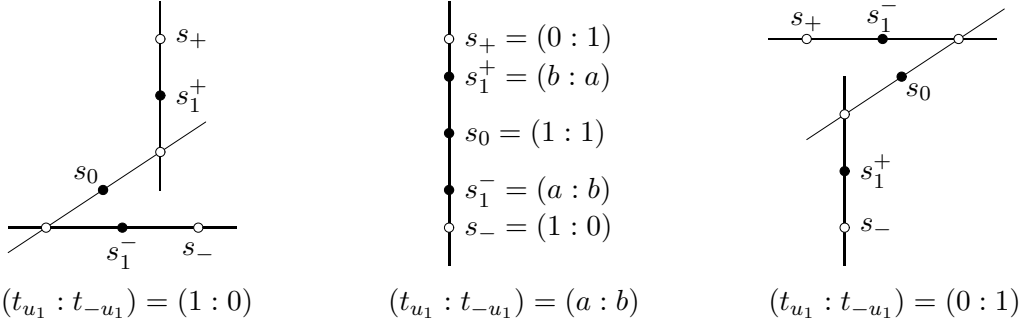
\medskip

\noindent
\begin{picture}(150,40)(0,0)
\put(4,0){\makebox(0,0)[l]{\small$(t_{u_1}:t_{-u_1})=(1:0)$}}
\put(5,5){
\put(20,35){\line(0,-1){25}}
\put(26,19){\line(-3,-2){26}}
\put(30,5){\line(-1,0){30}}
\filltype{white}
\put(20,30){\circle*{1.2}}
\put(22,30){\makebox(0,0)[l]{\small$s_+$}}
\put(25,5){\circle*{1.2}}
\put(25,3){\makebox(0,0)[t]{\small$s_-$}}
\put(20,15){\circle*{1.2}}
\put(5,5){\circle*{1.2}}
\filltype{black}
\put(20,22.5){\circle*{1.2}}
\put(22,22.5){\makebox(0,0)[l]{\small$s_1^+$}}
\put(15,5){\circle*{1.2}}
\put(15,4){\makebox(0,0)[t]{\small$s_1^-$}}
\put(12.5,10){\circle*{1.2}}
\put(12,12){\makebox(0,0)[r]{\small$s_0$}}
}

\put(55,0){\makebox(0,0)[l]{\small$(t_{u_1}:t_{-u_1})=(a:b)$}}
\put(63,5){
\put(0,0){\line(0,1){35}}
\filltype{white}
\put(0,30){\circle*{1.2}}
\put(2,30){\makebox(0,0)[l]{\small$s_+=(0:1)$}}
\put(0,5){\circle*{1.2}}
\put(2,5){\makebox(0,0)[l]{\small$s_-=(1:0)$}}
\filltype{black}
\put(0,17.5){\circle*{1.2}}
\put(2,17.5){\makebox(0,0)[l]{\small$s_0=(1:1)$}}
\put(0,25){\circle*{1.2}}
\put(2,25){\makebox(0,0)[l]{\small$s_1^+=(b:a)$}}
\put(0,10){\circle*{1.2}}
\put(2,10){\makebox(0,0)[l]{\small$s_1^-=(a:b)$}}
}

\put(105,0){\makebox(0,0)[l]{\small$(t_{u_1}:t_{-u_1})=(0:1)$}}
\put(90,5){
\put(25,0){\line(0,1){25}}
\put(46,34){\line(-3,-2){26}}
\put(45,30){\line(-1,0){30}}
\filltype{white}
\put(20,30){\circle*{1.2}}
\put(20,31){\makebox(0,0)[b]{\small$s_+$}}
\put(25,5){\circle*{1.2}}
\put(27,5){\makebox(0,0)[l]{\small$s_-$}}
\put(40,30){\circle*{1.2}}
\put(25,20){\circle*{1.2}}
\filltype{black}
\put(30,30){\circle*{1.2}}
\put(30,31){\makebox(0,0)[b]{\small$s_1^-$}}
\put(25,12.5){\circle*{1.2}}
\put(27,12.5){\makebox(0,0)[l]{\small$s_1^+$}}
\put(32.5,25){\circle*{1.2}}
\put(33,23){\makebox(0,0)[l]{\small$s_0$}}
}
\end{picture}
\end{ex}

\bigskip

By remark \ref{rem:emb-univ-B_n-curve} the universal $B_n$-curve
over $X(B_n)$ can be embedded into a product $P(B_{n+1}/B_n)_{X(B_n)}
\cong (\P^1)^{2n+1}_{X(B_n)}$ and the embedded curve is given by
equations (\ref{eq:univ-B_n-curveI}) determined by the universal
$B_n$-data.
We show that any $B_n$-curve $C$ over a field can be embedded into
a product $(\P^1)^{2n+1}$ and extract $B_n$-data such that $C$
is described by the same equations as the universal curve.

We fix the following notation: given a $B_n$-curve
$(C\to Y,I,s_-,s_+,s_0,s_1^\pm,\ldots,s_n^\pm)$ we associate
with the sections $s_0,s_i^-,s_i^+$ the roots
$u_{n+1},\alpha_i^-,\alpha_i^+$ of $B_{n+1}^+\setminus B_n^+
=\{u_{n+1},\alpha_1^\pm,\ldots,\alpha_n^\pm\}$ (cf.\ remark
\ref{rem:emb-univ-B_n-curve} and construction \ref{con:univ-B_n-curve});
we will write $\alpha_s$ for the positive root associated with the
section $s$ and conversely $s_\alpha$ for the section associated
with the root.

\begin{prop}\label{prop:emb-B_n-curve}
Let $(C,I,s_-,s_+,s_0,s_1^\pm,\ldots,s_n^\pm)$ be a $B_n$-curve
over a field.
For any $s\in\{s_0,s_1^\pm,\ldots,s_n^\pm\}$ let $z_{\alpha_s},
z_{-\alpha_s}\in H^0(C,\O_C(s))$ be a basis of $H^0(C,\O_C(s))$
such that $z_{-\alpha_s}(s_-)=0$, $z_{\alpha_s}(s_+)=0$,
$z_{-\alpha_s}(s)=z_{\alpha_s}(s)\neq 0$ (cf.\ remark
\ref{rem:emb-univ-B_n-curve} for this choice).
We will write $\P^1_{\{\pm\alpha_s\}}$ for $\P(H^0(C,\O_C(s)))$.
Then by
\[
(t_{\beta}:t_{-\beta})=(z_{-\alpha_2}(s_1):z_{\alpha_2}(s_1))
\]
if $\beta=\alpha_1-\alpha_2$ is a root of $B_n$
and $\alpha_1,\alpha_2$ are roots corresponding to distinct marked
points $s_1,s_2\in\{s_0,s_1^\pm,\ldots,s_n^\pm\}$, we can define
$B_n$-data $(t_{\beta}:t_{-\beta})_{\{\pm\beta\}\subseteq B_n}$
and the morphism
\[\textstyle
C\;\to\;\prod_{i=1}^{n}\P^1_{\{\pm\alpha_i^+\}}\times
\prod_{i=1}^{n}\P^1_{\{\pm\alpha_i^-\}}\times\P^1_{\{\pm u_{n+1}\}}
=P(B_{n+1}/B_n)
\]
is an isomorphism onto the closed subvariety $C'\subseteq P(B_{n+1}/B_n)$
determined by the homogeneous equations
\begin{equation}\label{eq:B_n-curve}
t_{\beta}z_{\alpha_2}z_{-\alpha_1}=t_{-\beta}z_{-\alpha_2}z_{\alpha_1}
\qquad
\begin{array}{l}
\textit{for}\;\alpha_1,\alpha_2\in \{u_{n+1},\alpha_1^\pm,
\ldots,\alpha_n^\pm\}\\
\textit{such that}\;\beta=\alpha_1-\alpha_2\;\textit{is a root of $B_n$}\\
\end{array}
\end{equation}
Furthermore, $C'$ together with the marked points $s_0'$ resp.\
${s_i'}^\pm$ defined by the additional equations
$z_{u_{n+1}}=z_{-u_{n+1}}$ resp.\
$z_{\alpha_i^\pm}=z_{-\alpha_i^\pm}$, the poles $s_-'$ resp.\ $s_+'$
defined by $z_{-u_{n+1}}=0$, $z_{-\alpha_i^\pm}=0$ $(i=1,\ldots,n)$
resp.\ $z_{u_{n+1}}=0$, $z_{\alpha_i^\pm}=0$ $(i=1,\ldots,n)$ and
the involution $I'$ given by $\:\P^1_{\{\pm\alpha_i^+\}}
\leftrightarrow\P^1_{\{\pm\alpha_i^-\}}$,
$z_{\alpha_i^+}\!\leftrightarrow \!z_{-\alpha_i^-}$ and
$\:\P^1_{\{\pm u_{n+1}\}}\leftrightarrow \P^1_{\{\pm u_{n+1}\}}$,
$z_{u_{n+1}}\!\leftrightarrow\!z_{-u_{n+1}}$ is a $B_n$-curve and
$(C,I,s_-,s_+,s_0,s_1^\pm,\ldots,s_n^\pm)\to
(C',I',s_-',s_+',s_0',{s_1'}^\pm,\ldots,{s_n'}^\pm)$ an isomorphism
of $B_n$-curves.
\end{prop}
\begin{proof}
The data $(t_{\beta}:t_{-\beta})$ is defined as position
of a marked point $s_1$ relative to another marked point $s_2$
of $C$ if $\beta=\alpha_{s_1}-\alpha_{s_2}$. We also write
$s_1/s_2$ for this data. We have the following cases:
\[
\begin{array}{ccccc}
\beta_{ij}&=&\alpha_i^+-\alpha_j^+&=&\alpha_j^--\alpha_i^-\\
\gamma_{ij}&=&\alpha_i^+-\alpha_j^-&=&\alpha_j^+-\alpha_i^-\\
u_i&=&\alpha_i^+-u_{n+1}&=&u_{n+1}-\alpha_i^-\\
\end{array}
\]
Note that because of the symmetry of the $B_n$-curve with respect
to the involution $I$ we have for the corresponding data
$s_i^+/s_j^+=s_j^-/s_i^-$, $s_i^+/s_j^-=s_j^+/s_i^-$,
$s_i^+/s_0=s_0/s_i^-$, so the data
$(t_{\beta}:t_{-\beta})_{\{\pm\beta\}\subseteq B_n}$
is well defined.

The rest of the proof is similar to the proof of
\cite[Prop.\ 3.12]{BB11}. To check that
$(t_{\beta}:t_{-\beta})_{\{\pm\beta\}\subseteq B_n}$
is $B_n$-data, we have to check the equations
$t_\beta t_\gamma t_{-\delta}=t_{-\beta}t_{-\gamma}t_{\delta}$
for the linear relations $\beta+\gamma=\delta$ given in lemma
\ref{le:linrel-B_n}; these equations can be written in the form
$s_1/s_2\cdot s_2/s_3=s_1/s_3$ for some sections $s_1,s_2,s_3$.
\end{proof}

We will continue to use the notations $s'/s=(t_\beta:t_{-\beta})$ for
$\beta=\alpha_{s'}-\alpha_{s}$, we have $s_-/s=(0:1)$ and
$s_+/s=(1:0)$ (points $s',s_-,s_+$ with respect to the coordinates
$(z_{-\alpha_{s}}:z_{\alpha_{s}})$).

\begin{lemma}\label{le:B_n-data--B_n-curve}
Any $B_n$-data over a field arises as $B_n$-data extracted from
a $B_n$-curve by the method of proposition \ref{prop:emb-B_n-curve}.
\end{lemma}
\begin{proof} Let $(t_{\beta}:t_{-\beta})_{\{\pm\beta\}\subseteq B_n}$
be $B_n$-data over a field.

We can define an ordering $\prec$ on the set of positive roots
$\{u_{n+1},\alpha_1^\pm,\ldots,\alpha_n^\pm\}=B_{n+1}^+\setminus B_n^+$:
for distinct $\alpha,\alpha'$ define $\alpha'\prec\alpha$
(resp.\ $\alpha'\preceq\alpha$) if $(t_{\beta}:t_{-\beta})=(0:1)$
(resp.\ $(t_{\beta}:t_{-\beta})\neq(1:0)$) for $\beta=\alpha'-\alpha$.
This defines a decomposition
$\{u_{n+1},\alpha_1^\pm,\ldots,\alpha_n^\pm\}=P_{-m}\sqcup\ldots
\sqcup P_m$ into nonempty equivalence classes such that
$\alpha'\prec\alpha$ $\Longleftrightarrow$ $\alpha'\in P_{k'}$,
$\alpha\in P_k$ for $k'<k$.
We have the symmetries $u_{n+1}\prec\alpha_i^\pm$
$\Longleftrightarrow$ $\alpha_i^\mp\prec u_{n+1}$ and
$\alpha_i^{\varepsilon_i}\prec\alpha_j^{\varepsilon_j}$
$\Longleftrightarrow$ $\alpha_j^{-\varepsilon_j}\prec
\alpha_i^{-\varepsilon_i}$ and these symmetries imply
$u_{n+1}\in P_0$ and $\alpha_i^+\in P_k$ $\Longleftrightarrow$
$\alpha_i^-\in P_{-k}$.

Now it is easy to construct a $B_n$-curve such that the $B_n$-data
extracted from it by the method of proposition \ref{prop:emb-B_n-curve}
is the given $B_n$-data by taking a chain of projective lines of
length $2m+1$ with involution $(C,I,s_-,s_+)$ (see definition
\ref{def:chaininv}) and choosing suitable marked points satisfying
$s_0\in C_0$ and $s_i^\pm\in C_k\Longleftrightarrow\alpha_i^\pm\in P_k$.
\end{proof}

Let $C$ be a $B_n$-curve over a field. It decomposes into irreducible
components $C=C_{-m}\cup\ldots\cup C_m$ with $s_-\in C_{-m}$,
$s_+\in C_m$. The decomposition
\[\{0,\pm1,\ldots,\pm n\}=P_{-m}\sqcup\ldots\sqcup P_m\]
such that $0\in P_0$ and $\varepsilon i\in P_k\Longleftrightarrow
s_i^\varepsilon\in C_k$, we will call the combinatorial type of
the $B_n$-curve (or of the corresponding $B_n$-data) over a field.
We will also write this in the form $s_{i_1}^{-\varepsilon_1}\ldots
s_{i_l}^{-\varepsilon_l}|\ldots|s_{i_1}^{\varepsilon_1}\ldots
s_{i_l}^{\varepsilon_l}$ with the sections for the different sets
$P_k$ separated by the symbol "$|\,$" starting on the left with $P_{-m}$.
Considering the fibres of the universal $B_n$-curve resp.\ the universal
$B_n$-data, these combinatorial types determine a stratification of
$X(B_n)$ which coincides with the stratification of this toric variety
into torus orbits.

\begin{prop}\label{prop:combtypeBn}
Over the torus orbit in $X(B_n)$ corresponding to the one-dimen-sional
cone generated by $\varepsilon_{i_1}v_{i_1}+\ldots+
\varepsilon_{i_k}v_{i_k}$ we have the combinatorial type
\[s_{i_1}^{\varepsilon_{i_1}}\cdots\;s_{i_k}^{\varepsilon_{i_k}}|
s_0s_{i_{k+1}}^\pm\cdots\;s_{i_n}^\pm|
s_{i_1}^{-\varepsilon_{i_1}}\cdots\;s_{i_k}^{-\varepsilon_{i_k}}\]
\end{prop}
\begin{proof}
The universal $B_n$-data over each point of the closure of the orbit
corresponding to a generator of a one-dimensional cone generated by
$v$ has the property $(t_{\beta}:t_{-\beta})=(0:1)$ if
$\langle\beta,v\rangle>0$ (see \cite[Rem.\ 1.21]{BB11}).
For $v=\varepsilon_{i_1}v_{i_1}+\ldots+\varepsilon_{i_k}v_{i_k}$
this in particular implies
$s_{i_l}^{\varepsilon_{i_l}}/s_0=s_0/s_{i_l}^{-\varepsilon_{i_l}}
=(t_{\varepsilon_{i_l}u_{i_l}}:t_{-\varepsilon_{i_l}u_{i_l}})=(0:1)
=s_-/s_0$.
\end{proof}

\medskip
\section{Moduli space of $B_n$-curves}
\label{sec:modspace}

In this section we show that there is a fine moduli space of
$B_n$-curves $\overline{L}_n^{0,\pm}$ which is isomorphic to the
toric variety $X(B_n)$ by constructing an isomorphism between
the moduli functor of $B_n$-curves and the functor of $X(B_n)$.
For the second functor we use the description in \cite[1.3]{BB11}
in terms of $B_n$-data.

\medskip

To relate $B_n$-curves to $B_n$-data we consider an embedding
of arbitrary $B_n$-curves over a scheme $Y$ into a product
$(\P^1)_Y^{2n+1}$ that generalises the embedding in proposition
\ref{prop:emb-B_n-curve} to the relative situation.
The main tool are the following contraction morphisms
(cf.\ \cite[3.3]{BB11}): for a subset $\{s_1,\ldots,s_l\}$ of
the sections of a pointed chain of projective lines $C$ there
is a line bundle $\O_C(s_1+\ldots+s_l)$ on $C$ and a morphism
$C\to C_{\{s_1,\ldots,s_l\}}\subseteq
\P_Y(\pi_*\mathscr\O_C(s_1+\ldots+s_l))$ such that the
morphisms $C_y\to(C_{\{s_1,\ldots,s_l\}})_y$ on the fibres
are isomorphisms on the components containing one of the
sections $s_i(y)$ and contract all other components
(see \cite[Constr.\ 3.15]{BB11}).

We will make use of the particular cases of contraction with respect
to one section onto a $\P^1$-bundle, with respect to two sections
onto an $A_1$-curve and with respect to three sections onto an
$A_2$-curve; we will apply \cite[Constr.\ 3.16; Lemma 3.17 and 3.18]{BB11}.

\medskip

We associate with the sections $s_0,s_i^\pm$ the roots
$u_{n+1},\alpha_i^\pm$ as we did before proposition
\ref{prop:emb-B_n-curve}.
For a $B_n$-curve $(C\to Y,I,s_-,s_+,s_1^\pm,\ldots,s_n^\pm)$
we denote the contraction morphisms with respect to one section
$s_0,s_i^-$ resp.\ $s_i^+$ by $p_0\colon C\to(\P^1_{\{\pm u_{n+1}\}})_Y$,
$p_i^-\colon C\to(\P^1_{\{\pm\alpha_i^-\}})_Y$ resp.\
$p_i^+\colon C\to (\P^1_{\{\pm\alpha_i^+\}})_Y$, where
$(\P^1_{\{\pm u_{n+1}\}})_Y$, $(\P^1_{\{\pm\alpha_i^-\}})_Y$ resp.\
$(\P^1_{\{\pm\alpha_i^+\}})_Y$ is a copy of $\P^1_Y$ with
homogeneous coordinates $z_{u_{n+1}},z_{-u_{n+1}}$ resp.\
$z_{\alpha_i^-},z_{-\alpha_i^-}$ resp.\ $z_{\alpha_i^+},z_{-\alpha_i^+}$
such that in these coordinates $s_-,s_+,s_0$ resp.\ $s_-,s_+,s_i^-$
resp.\ $s_-,s_+,s_i^+$ become the $(1:0),(0:1),(1:1)$-section of $\P^1_Y$.

\begin{thm}
There exists a fine moduli space $\overline{L}_n^{0,\pm}$ of $B_n$-curves
isomorphic to the toric variety $X(B_n)$ with universal family
$X(B_{n+1})\to X(B_n)$.
\end{thm}
\begin{proof}
We show that the moduli functor of $B_n$-curves
$\underline{\overline{L}_n^{0,\pm}}$ as defined in section
\ref{sec:moduliprobl} is\linebreak isomorphic to the functor 
$F_{B_n}$ of the toric variety $X(B_n)$ as described in \cite[1.3]{BB11}.

\smallskip

Let $Y$ be a scheme. For $B_n$-data on $Y$ we construct a
$B_n$-curve $C$ over $Y$ via equations in $P(B_{n+1}/B_n)_Y$
as in remark \ref{rem:emb-univ-B_n-curve} with the given
$B_n$-data on $Y$ replacing the universal $B_n$-data on $X(B_n)$.
This is a $B_n$-curve: any $B_n$-data is pull-back of the
universal $B_n$-data on $X(B_n)$, so the constructed curve
is pull-back of the universal $B_n$-curve over $X(B_n)$.

In the other direction, given a $B_n$-curve over $Y$ we
extract $B_n$-data. For each pair of distinct sections
$s_1,s_2\in\{s_0,s_1^\pm,\ldots,s_n^\pm\}$ we have a
contraction morphism $C\to C_{\{s_1,s_2\}}$ onto an
$A_1$-curve over $Y$. From $(C_{\{s_1,s_2\}},s_1,s_2)$ we
extract $A_1$-data $(\mathscr L_{\{1,2\}},\{t_{1,2},t_{2,1}\})$
as in \cite[Constr.\ 3.16]{BB11}:
we put $\mathscr L_{\{\pm\beta\}}:=\mathscr L_{\{1,2\}}$,
$t_{\beta}:=t_{1,2}$, $t_{-\beta}:=t_{2,1}$ for
$\beta=\alpha_{s_1}-\alpha_{s_2}$ (then $(t_{\beta}:t_{-\beta})$
measures the position of $s_1$ relative to $s_2$, we write this
as $s_1/s_2$). We have the following cases:
\[
\begin{array}{ccccc}
\beta_{ij}&=&\alpha_i^+-\alpha_j^+&=&\alpha_j^--\alpha_i^-\\
\gamma_{ij}&=&\alpha_i^+-\alpha_j^-&=&\alpha_j^+-\alpha_i^-\\
u_i&=&\alpha_i^+-u_{n+1}&=&u_{n+1}-\alpha_i^-\\
\end{array}
\]
Because of the symmetry of the $B_n$-curve with respect
to the involution we have for the corresponding data
$s_i^+/s_j^+=s_j^-/s_i^-$, $s_i^+/s_j^-=s_j^+/s_i^-$,
$s_i^+/s_0=s_0/s_i^-$, so the data $(\mathscr L_{\{\pm\beta\}},
\{t_{\beta},t_{-\beta}\})_{\{\pm\beta\}\subseteq B_n}$ is well
defined.

We show that the data obtained this way is $B_n$-data. Let
$\beta,\gamma,\delta$ be positive roots of $B_n$ such that
$\beta+\gamma=\delta$. We have to verify that the collection
$\{(\mathscr L_{\{\pm\beta\}},\{t_{\beta},t_{-\beta}\}),$ $
(\mathscr L_{\{\pm\gamma\}},\{t_{\gamma},t_{-\gamma}\}),
(\mathscr L_{\{\pm\delta\}},\{t_{\delta},t_{-\delta}\})\}$
satisfies $t_{\beta}t_{\gamma}t_{-\delta}=t_{-\beta}t_{-\gamma}
t_{\delta}$, which means that it is $A_2$-data.
By lemma \ref{le:linrel-B_n} we have the following cases:
\[
\begin{array}{cl}
\beta_{ij}+u_j=u_i&\quad (i,j\in\{1,\ldots,n\},\:i<j)\\
u_i+u_j=\gamma_{ij}&\quad (i,j\in\{1,\ldots,n\},\:i\neq j)\\
\beta_{ij}+\beta_{jk}=\beta_{ik}&\quad (i,j,k\in\{1,\ldots,n\},\:i<j<k)\\
\beta_{ij}+\gamma_{jk}=\gamma_{ik}&\quad (i,j,k\in\{1,\ldots,n\},\:i<j,\:
k\neq i,j)\\
\end{array}
\]
In each of these cases we can write $\beta=\alpha_{s_1}-\alpha_{s_2}$,
$\gamma=\alpha_{s_2}-\alpha_{s_3}$ for three distinct sections
$s_1,s_2,s_3\in\{s_0,s_1^\pm,\ldots,s_n^\pm\}$. Then these equations
can be interpreted as relations between the relative positions of pairs
of sections in a set of three sections, we write this as
$s_1/s_2\cdot s_2/s_3=s_1/s_3$:
\[
\begin{array}{ccc}
\beta_{ij}+u_j=u_i,&\quad \beta_{ij}=\alpha_i^+-\alpha_j^+,
u_j=\alpha_j^+-u_{n+1},&\quad s_i^+/s_j^+\cdot s_j^+/s_0=s_i^+/s_0\\
u_i+u_j=\gamma_{ij},&\quad u_i=\alpha_i^+-u_{n+1},u_j=u_{n+1}-\alpha_j^-,&
\quad s_i^+/s_0\cdot s_0/s_j^-=s_i^+/s_j^-\\
\beta_{ij}+\beta_{jk}=\beta_{ik},&\quad \beta_{ij}=\alpha_i^+-\alpha_j^+,
\beta_{jk}=\alpha_j^+-\alpha_k^+,&\quad s_i^+/s_j^+\cdot s_j^+/s_k^+
=s_i^+/s_k^+\\
\beta_{ij}+\gamma_{jk}=\gamma_{ik},&\quad \beta_{ij}=\alpha_i^+-\alpha_j^+,
\gamma_{jk}=\alpha_j^+-\alpha_k^-,&\quad s_i^+/s_j^+\cdot s_j^+/s_k^-
=s_i^+/s_k^-\\
\end{array}
\]
We have a contraction morphism $C\to C_{\{s_1,s_2,s_3\}}$ over $Y$
onto an $A_2$-curve $C_{\{s_1,s_2,s_3\}}$ over $Y$.
The data $\{(\mathscr L_{\{\pm\beta\}},\{t_{\beta},t_{-\beta}\}),
(\mathscr L_{\{\pm\gamma\}},\{t_{\gamma},t_{-\gamma}\}),\linebreak
(\mathscr L_{\{\pm\delta\}},\{t_{\delta},t_{-\delta}\})\}$
coincides with the data extracted from this $A_2$-curve and is
$A_2$-data by \cite[Lemma 3.18]{BB11}.

\smallskip

Both constructions commute with base-change and thus define
morphisms of functors $F_{B_n}\to\underline{\overline{L}_n^{0,\pm}}$
and $\underline{\overline{L}_n^{0,\pm}}\to F_{B_n}$.
As in the proof of \cite[Thm.\ 3.19]{BB11} one shows that they are
inverse to each other.
\end{proof}

\begin{rem}
The moduli space $\overline{L}_n^{0,\pm}$ embeds naturally into
$\overline{L}_{2n+1}$. A morphism $\overline{L}_n^{0,\pm}\to
\overline{L}_{2n+1}$ is given by considering a $B_n$-curve with
sections $s_1^-,\ldots,s_n^-,s_0,$\linebreak $s_n^+,\ldots,s_1^+$ as 
an $A_{2n}$-curve with sections $s_1,\ldots,s_{n+1},\ldots,s_{2n+1}$.
This corresponds to the toric morphism $X(B_n)\to X(A_{2n})$ given
by the projection of root systems $A_{2n}\to B_n$ mapping
$u_i-u_{n+1}\mapsto u_i$, $u_{2n+2-i}-u_{n+1}\mapsto-u_i$
$(i=1,\ldots,n)$ with kernel generated by $u_i+u_{2n+2-i}-2u_{n+1}$
$(i=1,\ldots,n)$.
\end{rem}

\medskip
\section{(Co)homology of $\overline{L}_n^{0,\pm}=X(B_n)$}
\label{sec:cohomBn}

We show that the (co)homology of the moduli space
$\overline{L}_n^{0,\pm}=X(B_n)$ over the complex numbers has a description
similar to that of the (co)homology of the Losev-Manin moduli spaces
$\overline{L}_n=X(A_n)$ (cf.\ \cite[2.2]{BB11}).

The torus invariant divisors of $\overline{L}_n^{0,\pm}=X(B_n)$
correspond to elements of the set $\mathcal B$
(see section \ref{sec:X(Bn)} and prop.\ \ref{prop:combtypeBn}).
Here, as in the case of the toric varieties $X(A_n)$,
all primitive collections consist of two elements corresponding to
non comparable sets $B,B'\in\mathcal B$.
As usual the integral cohomology is torsion free and confined to
the even degrees and standard methods from toric geometry
(see e.g.\ \cite[(10.8)]{Dan}) give:

\begin{prop}
For the cohomology ring of the toric variety
$X(B_n)$ over the complex numbers we have
\[H^*(X(B_n),\Z)\;\cong\;\Z[\,l_B:B\in\mathcal B\,]/(R_1+R_2)\]
where $R_1$ is the ideal generated by the elements
$r_i=\sum_{i\in B}l_B-\sum_{-i\in B}l_B$ for $i=1,\ldots,n$ and
$R_2$ the ideal generated by the elements $r_{B,B'}=l_Bl_{B'}$
for $B,B'\in\mathcal B$ such that $B\not\subseteq B'$,
$B'\not\subseteq B$.
\end{prop}

We proceed by determining the Betti numbers and the
Poincar\'e polynomial and obtain the following closed formula
which is an analogue to \cite[(2.3)]{LM00}.

\begin{prop}
Let $p_{X(B_n)}(t)=\sum_{i=0}^n\beta_{2i}(X(B_n))t^i$ be the Poincar\'e
polynomial of $X(B_n)$ with $\beta_{2i}(X(B_n))=\rk H^{2i}(X(B_n),\Z)$
the Betti numbers. Then we have
\[
\sum_{n=0}^\infty\frac{p_{X(B_n)}(t)}{n!}y^n
\;=\;e^{y(t-1)}\frac{t-1}{t-e^{2y(t-1)}}
\;\in\;\Z[t][[y]]
\]
\end{prop}
\begin{proof}
We have $p_{X(B_n)}(t)=\sum_{m=0}^nd_m(B_n)(t-1)^{n-m}$ (see
\cite[p.\ 92]{Ful} or \cite[(10.8)]{Dan}; this can be shown in
different ways, one possibility is by counting points over finite
fields as in \cite{LM00}) with $d_m(B_n)=\textit{number of
$(n-m)$-dim.\ torus orbits of $X(B_n)$}=
\textit{number of $m$-dim.\ cones of $\Sigma(B_n)$}$. 
Inserting this into $\sum_{n=0}^\infty\frac{p_{X(B_n)}(t)}{n!}y^n$ 
and interchanging summation by $n$ and $m$, we get
\[\textstyle
\sum\limits_{n=0}^\infty\frac{p_{X(B_n)}(t)}{n!}y^n\;=\;
\sum\limits_{m=0}^\infty\frac{1}{(t-1)^m}\sum\limits_{n=m}^\infty
\frac{d_m(B_n)}{n!}(t-1)^ny^n
\]
The number $d_m(B_n)$ can be calculated as
\[\textstyle
\frac{1}{n!}d_m(B_n)\quad=\sum\limits_{(a_0,a_1,\ldots,a_m)}
\frac{1}{a_0!}\frac{2^{a_1}}{a_1!}\cdots\frac{2^{a_m}}{a_m!}
\]
where the sum runs over sequences $a_0\in\Z_{\geq0}$, $a_1\in\Z_{>0}$,
$\ldots$, $a_m\in\Z_{>0}$ such that $\sum_i a_i=n$ (note that any family
$B^{(m)}\subsetneq\ldots\subsetneq B^{(1)}$ of elements of $\mathcal B$
corresponding to an $m$-dimensional cone of $\Sigma(B_n)$ determines
such a partition by  $a_m=|B^{(m)}|$, $a_{m-1}=|B^{(m-1)}|-|B^{(m)}|$,
$\ldots$ , $a_0=n-|B^{(1)}|$, in addition we have orderings and signs).
Making use of the fact that $\frac{1}{n!}d_m(B_n)$ coincides with the
coefficient of $x^n$ in the power series $e^x(e^{2x}-1)^m$, we obtain
\[\textstyle
\sum\limits_{n=0}^\infty\frac{p_{X(B_n)}(t)}{n!}y^n\;=\;
e^{y(t-1)}\sum\limits_{m=0}^\infty\frac{1}{(t-1)^m}(e^{2y(t-1)}-1)^m
\]
which yields the result.
\end{proof}

In particular we have $\chi(X(B_n))=2^nn!$ (this reflects the fact
that we have $2^nn!$ maximal cones), $\beta_2(X(B_n))=3^n-n-1$
(corresponding to the fact that we have $3^n-1$ one-dimensional cones)
and for the first Poincar\'e  polynomials
\[
\begin{array}{c}
p_{X(B_1)}(t)=t+1,\quad p_{X(B_2)}(t)=t^2+6t+1,\quad
p_{X(B_3)}(t)=t^3+23t^2+23t+1\\
p_{X(B_4)}(t)=t^4+76t^3+230t^2+76t+1\\
p_{X(B_5)}(t)=t^5+237t^4+1682t^3+1682t^2+237t+1\\
p_{X(B_6)}(t)=t^6+722t^5+10543t^4+23548t^3+10543t^2+722t+1\\
\end{array}
\]

The ring $\Z[\,l_B:B\in\mathcal B\,]/R_2$ is the Stanley-Reisner
ring for the triangulation of the $(n-1)$-dimensional sphere
determined by the fan $\Sigma(B_n)$. It is a Cohen-Macaulay
ring and the elements $r_1,\ldots,r_n$ that generate $R_1$
form a regular sequence. The calculation of the Poincar\'e
polynomial of a toric variety in \cite[(10.8)]{Dan} in terms of
the numbers of cones of dimension $d=1,\ldots,n$ only depends
on the Hilbert-Poincar\'e series of the Stanley-Reisner ring of
the fan and the fact that the quotient by an ideal generated by
a regular sequence is taken.
In \cite{Re01} a ring has been defined by taking the same Stanley-Reisner
ring (over a field) but instead of $R_1$ an ideal generated by a
different regular sequence, so by construction this ring has the
same Poincar\'e polynomial as the cohomology ring of $X(B_n)$.

\medskip

The $\Z$-module $\Z[\,l_B:B\in\mathcal B\,]/(R_1+R_2)$ is generated
by the classes of square-free monomials (see \cite[(10.7.1)]{Dan}).
We can restrict to monomials each of which has only factors
corresponding to one-dimensional faces of one maximal cone.
Such a monomial $\prod_{i=1}^ml_{B^{(i)}}$ corresponds to an
$m$-dimensional face of the respective maximal cone and on the
other hand to a collection $B^{(m)}\subsetneq\ldots\subsetneq B^{(1)}$
of elements of $\mathcal B$. We denote the $\Z$-submodule of
$\Z[\,l_B:B\in\mathcal B\,]$ generated by these monomials by $G$.
There is the canonical isomorphism of $\Z$-modules
$G/U\cong\Z[\,l_B:B\in\mathcal B\,]/(R_1+R_2)$ where
$U=(R_1+R_2)\cap G$. As usual, the module $G/U$ can be identified
with the homology module $H_*(X(B_n),\Z)$. The monomial
$\prod_{i=1}^ml_{B^{(i)}}$ then corresponds to the class of the
orbit closure for the cone determined by the collection
$B^{(m)}\subsetneq\ldots\subsetneq B^{(1)}$, in particular the
monomials of $G$ of degree $m$ generate $H_{2(n-m)}(X(B_n),\Z)$.

\medskip

The maximal cones of the fan $\Sigma(B_n)$ correspond to collections
$B^{(n)}\subsetneq\ldots\subsetneq B^{(1)}$ of elements of $\mathcal B$
and these correspond to so called signed permutations, that is elements
of the Weyl group $W(B_n)=(\Z/2\Z)^n\rtimes S_n=:S_n^\pm$. A signed
permutation $w\in S_{n}^\pm$ corresponds via $(w(1),\ldots,w(n))$ to
a sequence of distinct elements in $\{\pm1,\ldots,\pm n\}$
for any $i$ not containing both $-i$ and $i$.
For a collection $B^{(n)}\subsetneq\ldots\subsetneq B^{(1)}$ of elements
of $\mathcal B$ the corresponding signed permutation $\sigma\in
S_n^\pm$ is given by $\{w(k)\}=B^{(k)}\setminus B^{(k+1)}$ for
$k=1,\ldots,n$ (put $B^{(n+1)}=\emptyset$).
The descent set of a signed permutation $w\in S_n^\pm$ is the set
(put $w(0)=0$)
\[\Desc(w)=\{k\in\{1,\ldots,n\}\:|\:w(k-1)>w(k)\}\]
For any $w\in S_n^\pm$ we define a monomial in $G$ by
\[\textstyle
l^w=\prod_{k\not\in\Desc(w)}l_{\{w(k),\ldots,w(n)\}}
\]
this way we have defined $2^nn!$ distinct monomials.

\begin{prop}
The classes of the monomials $l^w$ for $w\in S_n^\pm$
form a basis of the homology module $G/U=H_*(X(B_n),\Z)$. The module
of relations $U$ is generated by the elements
\[\textstyle
r_{i,j}((B^{(h)})_h,k)=\Big(\sum_{\tiny\,\hspace{-1mm}
\begin{array}{l}i\!\in\!B\\[-1mm]j\!\not\in\!B\\\end{array}
\hspace{-1.2mm}}\,l_B-
\sum_{\tiny\,\hspace{-1mm}
\begin{array}{l}j\!\in\!B\\[-1mm]i\!\not\in\!B\\\end{array}
\hspace{-1.2mm}}\,l_B\Big)
\prod_{h=1}^ml_{B^{(h)}}
\]
(sums over sets $B^{(k+1)}\subsetneq B\subsetneq B^{(k)}$) for
collections $B^{(m)}\subsetneq\ldots\subsetneq B^{(1)}$, $m\geq 1$
of elements of $\mathcal B$ and $k\in\{1,\ldots,m\}$, $i,j\in B^{(k)}
\setminus B^{(k+1)}$ (put $B^{(m+1)}=\emptyset$), $i\neq j$,
and by the elements
\[\textstyle
r_i((B^{(h)})_h)=\Big(\sum_{i\in B}l_B-\sum_{-i\in B}l_B\Big)
\prod_{h=1}^ml_{B^{(h)}}
\]
(sums over sets $B\in\mathcal B$ such that $B^{(1)}\subsetneq B$ if
$m\geq 1$) for collections $B^{(m)}\subsetneq\ldots\subsetneq B^{(1)}$,
$m\geq 0$ of elements of $\mathcal B$ and $i\in\{1,\ldots,n\}$ such that
$-i,i\not\in B^{(1)}$ if $m\geq 1$.
\end{prop}
\begin{proof}
We observe that the given relations are contained in $U$.
We have $2^nn!$ monomials $l^w$, this number coincides with
the rank of $G/U$. Thus it remains to show that every monomial in
$G$ via the given relations is equivalent to a linear combination
of the monomials $l^w$.

For a monomial $\prod_{k=1}^ml_{B^{(k)}}$ corresponding to a collection
$B^{(m)}\subsetneq\ldots\subsetneq B^{(1)}$, $m\geq1$ we define the
number $d(\,\prod_{k=1}^ml_{B^{(k)}}):=|\{k\in\{1,\ldots,m\}\:|\:
\min P_{k-1}>\max P_{k}\}|\in\Z_{\geq 0}$ in terms of the associated
partition $P_m=B^{(m)}$, $P_{m-1}=B^{(m-1)}\setminus B^{(m)}$, $\ldots\,$,
$P_1=B^{(1)}\setminus B^{(2)}$, $P_0=\{0,\pm1,\ldots,\pm n\}
\setminus\{\pm i\:|\:i\!\in B^{(1)}\;\textit{or}\;-i\!\in\!B^{(1)}\}$.
The monomials $y\in G$ satisfying $d(y)=0$ are exactly the monomials
of the form $l^w$.
We define the following ordering $\prec$ of the monomials of $G$:
take the partition $(P_k)_{k=0,\ldots,m}$ associated with a monomial
and consider the sequence that arises by taking the sets
$P_m,\ldots,P_1$ in this order and by ordering the elements
of each $P_k$ according to their size, on these sequences we take
the lexicographic order.

We show that every monomial in $G$ modulo $U$ is equivalent to a
linear combination of the monomials $l^w$, $w\in S_n^\pm$
by showing that every monomial $y\in G$ with $d(y)>0$ modulo a
relation is equivalent to a linear combination of monomials $y'$
with $y\prec y'$.
In fact, let $B^{(m)}\subsetneq\ldots\subsetneq B^{(1)}$, $m\geq 1$
be a collection of elements of $\mathcal B$ (put $B^{(m+1)}:=\emptyset$)
with associated partition $(P_k)_{k=0,\ldots,m}$ such that the
corresponding monomial $y=\prod_{k=1}^ml_{B^{(k)}}$ satisfies $d(y)>0$.
Take $k\in\{1,\ldots,m\}$ such that $i:=\min P_{k-1}>\max P_k=:j$.
If $k\in\{2,\ldots,m\}$ then
\[\textstyle
r_{i,j}((B^{(h)})_{h\neq k},k-1)=
\Big(\sum_{\tiny\hspace{-1mm}
\begin{array}{l}i\!\in\!B\\[-1mm]j\!\not\in\!B\\\end{array}
\hspace{-1.2mm}}l_B-
\sum_{\tiny\hspace{-1mm}
\begin{array}{l}j\!\in\!B\\[-1mm]i\!\not\in\!B\\\end{array}
\hspace{-1.2mm}}l_B\Big)
\prod_{h\neq k}l_{B^{(h)}}
\]
(sums over sets $B$ such that $B^{(k+1)}\subsetneq B\subsetneq B^{(k-1)}$)
is a relation that contains $y$ as the unique monomial minimal with
respect to $\prec$. If $k=1$ then
\[\textstyle
r_{-j}((B^{(h)})_{h\neq 1})=\Big(\sum_{-j\in B}l_B-\sum_{j\in B}l_B\Big)
\prod_{h=2}^{m}l_{B^{(h)}}
\]
(sums over sets $B\in\mathcal B$ such that $B^{(2)}\subsetneq B$)
is such a relation.
\end{proof}

The proposition implies that the Betti numbers of $X(B_n)$ coincide
with the number of signed permutations with prescribed number of
descents, for this see also \cite[Section 4]{DL94}, \cite{St94}.
Our basis of $H_*(X(B_n),\Z)$ coincides with the basis given
in \cite{Kl85}, \cite{Kl95} in the general case of toric varieties
associated with root systems (see the following remark).

\begin{rem}
In \cite{Kl85} a basis of the homology $H_*(X(R),\Z)$ is constructed
as follows. For a fixed set of simple roots $S\subset R$ and the
corresponding Weyl chamber $\sigma_S=S^\vee$ consider for each $w\in W(R)$
the face $\sigma_w\subseteq w\sigma_S$ given as the intersection of those
walls of $w\sigma_S$ that separate $\sigma_S$ and $w\sigma_S$, i.e.\
we have the intersection of $w\sigma_S$ with those subspaces
$(w\alpha)^\bot$, $\alpha\in S$, for which $w\alpha$ is a negative root.
The cycles corresponding to the family of cones $(\sigma_w)_{w\in W(R)}$
form a basis of $H_*(X(R),\Z)$.

In our case we may choose the set of simple roots
$S=\{u_n-u_{n-1},\ldots,u_2-u_1,u_1\}\subset B_n$;
the corresponding Weyl chamber is generated by
$v_n,v_{n-1}+v_n,\ldots,v_1+\ldots+v_n$.
Then for $w\in W(B_n)=S_n^\pm$ we have
$w(u_k-u_{k-1})\;\textit{is negative}$ $\Longleftrightarrow$
$w(k-1)>w(k)$ for $k\in\{2,\ldots,n\}$ and $w(u_1)\;\textit{is negative}$
$\Longleftrightarrow$ $0>w(1)$. So, each root $\alpha\in S$ such that
$w\alpha$ is negative corresponds to an element of $\Desc(w)$. Since
$(w(u_k-u_{k-1}))^\bot\cap w\sigma_S$ is generated by
$\{w(v_n),\ldots,w(v_1+\ldots+v_n)\}\setminus\{w(v_k+\ldots+v_n)\}$
and $(w(u_1))^\bot\cap w\sigma_S$ by
$\{w(v_n),\ldots,w(v_2+\ldots+v_n)\}$, it follows that $\sigma_w$
is generated by $\{v_{\{w(k),\ldots,w(n)\}}\:|\:k\not\in\Desc(w)\}$
and the class of the respective torus invariant cycle corresponds
to the monomial $l^w$.
\end{rem}

\bigskip
\section{Root systems of type $C$}
\label{sec:C_n}

Consider an $n$-dimensional Euclidean space $E$ with basis
$u_1,\ldots,u_n$. The root system $C_n$ in $E$ consists of the
$2n^2$ roots:
\[
\pm 2u_i\;\:\textit{for}\;\:i\in\{1,\ldots,n\};\quad\pm(u_i+u_j),
\pm(u_i-u_j)\;\:\textit{for}\;\:i,j\in\{1,\ldots,n\},i<j.
\]
The following is a set of simple roots:
\[u_1-u_2,u_2-u_3,\ldots,u_{n-1}-u_n,2u_n.\]
Let $M(C_n)$ be the root lattice.
The Weyl group $(\Z/2\Z)^n\rtimes S_n$ acts by $u_i\mapsto\pm u_i$
and by permuting the $u_i$.
So there are $2^nn!$ sets of simple roots, these are of the form
$\varepsilon_1u_{i_1}-\varepsilon_2u_{i_2},
\varepsilon_2u_{i_2}-\varepsilon_3u_{i_3},\ldots,
\varepsilon_{n-1}u_{i_{n-1}}-\varepsilon_{n}u_{i_{n}},
2\varepsilon_nu_{i_n}$ for orderings $i_1,\ldots,i_n$ of the set
$\{1,\ldots,n\}$ and signs $\varepsilon_1,\ldots,\varepsilon_n$.

\medskip

The vector space $E^*$ dual to $E$ with basis $v_1,\ldots,v_n$
dual to $u_1,\ldots,u_n$ contains the lattice $N(C_n)$ dual to
$M(C_n)$. To describe the fan $\Sigma(C_n)$ in the lattice $N(C_n)$
we describe a Weyl chamber. For the set of simple roots
$S=\{u_1-u_2,u_2-u_3,\ldots,u_{n-1}-u_n,2u_n\}$ has the dual basis
$v_1,v_1+v_2,\ldots,v_1+\ldots+v_{n-1},\frac{1}{2}(v_1+\ldots+v_n)$
of $N(C_n)$, the Weyl chamber $\sigma_S$ is equal to
$\langle v_1,v_1+v_2,\ldots,v_1+\ldots+v_{n-1},
\frac{1}{2}(v_1+\ldots+v_n)\rangle_{\Q_{\geq0}}$.
All Weyl chambers are generated by collections of elements of the form
$\varepsilon_1v_{i_1},\varepsilon_1v_{i_1}+\varepsilon_2v_{i_2},
\ldots,\frac{1}{2}(\varepsilon_1v_{i_1}+\ldots+\varepsilon_nv_{i_n})$
for orderings $i_1,\ldots,i_n$ of the set $\{1,\ldots,n\}$ and signs
$\varepsilon_i$.
There are $3^n-1$ one-dimensional cones generated by elements of
the form $\varepsilon_1v_{i_1}+\ldots+\varepsilon_kv_{i_k}$
for $k\in\{1,\ldots,n-1\}$ or of the form
$\frac{1}{2}(\varepsilon_1v_1+\ldots+\varepsilon_nv_n)$.

\medskip

The torus invariant divisor for the one-dimensional cone generated by
$\varepsilon_1v_{i_1}+\ldots+\varepsilon_kv_{i_k}$
is isomorphic to $X(C_{n-k})\times X(A_{k-1})$, that for
$\frac{1}{2}(\varepsilon_1v_1+\ldots+\varepsilon_nv_n)$
is isomorphic to $X(A_{n-1})$.

\medskip
\medskip
\noindent
{\bf\boldmath $X(C_{n+1})$ over $X(C_n)$.}
Consider the proper surjective morphism $X(C_{n+1})\to X(C_n)$ induced
by the root subsystem $C_n\subset C_{n+1}$ consisting of the roots
in the subspace generated by $u_1,\ldots,u_n$. As in the $B$-case
one shows that $X(C_{n+1})$ is flat over $X(C_n)$.

\medskip

The automorphism of $C_{n+1}$ given as the reflection for the root
$\pm u_{n+1}$ fixes $C_n\subset C_{n+1}$ and induces an involution
$I$ of $X(C_{n+1})$ over $X(C_n)$. We have two sections $s_-,s_+$
defined as in the $B$-case.
There are $2n+1$ additional pairs of opposite roots, the pairs
$\pm\alpha_i^+=\pm(u_{n+1}+u_i)$, $\pm\alpha_i^-=\pm(u_{n+1}-u_i)$
for $i\in\{1,\ldots,n\}$ and the pair $\pm 2u_{n+1}$.
Any pair $\pm\alpha_i^+$, $\pm\alpha_i^-$ defines a projection
onto the root subsystem $C_n\subset C_{n+1}$ in the sense of
\cite[1.2]{BB11}, thus we have sections $s_i^+$ and $s_i^-$.
The pair $\pm 2u_{n+1}$ does not define a projection of root systems
$C_{n+1}\to C_n$, so it does not induce a section. However, we can
consider the morphism $X(C_{n+1})\to\P^1_{\{\pm2u_{n+1}\}}$ and the
preimage of the point $(1:1)$. We denote this subscheme of $X(C_{n+1})$
by $S_0$; it is finite flat of degree $2$ over $X(C_n)$ (see below),
such a subscheme we will call a double-section.

\medskip

If we consider $X(C_{n+1})$ and $X(C_n)$ as embedded $X(C_{n+1})
\subseteq P(C_{n+1})$, $X(C_n)\subseteq P(C_n)$, then the morphism
$X(C_{n+1})\to X(C_n)$ is induced by the projection onto the
subproduct $P(C_{n+1})\to P(C_n)$ and $X(C_{n+1})$ is given
in $P(C_{n+1}/C_n)_{X(C_n)}=\big(\prod_{i=1}^n\P^1_{\{\pm\alpha_i^+\}}
\times\prod_{i=1}^n\P^1_{\{\pm\alpha_i^-\}}\times
\P^1_{\{\pm 2u_{n+1}\}}\big)_{X(C_n)}$ by the homogeneous equations
involving the universal $C_n$-data on $X(C_n)$
\begin{equation}\label{eq:CnI}
z_{\alpha_i^-}z_{\alpha_i^+}z_{-2u_{n+1}}
=z_{-\alpha_i^-}z_{-\alpha_i^+}z_{2u_{n+1}},
\quad i\in\{1,\ldots,n\}
\end{equation}
\vspace{-8mm}
\begin{equation}\label{eq:CnII}
t_{\beta}z_{\alpha_2}z_{-\alpha_1}=t_{-\beta}z_{-\alpha_2}z_{\alpha_1},
\quad \alpha_1,\alpha_2\in \{\alpha_1^\pm,\ldots,\alpha_n^\pm\},
\;\alpha_1\neq\alpha_2,\;\beta=\alpha_1-\alpha_2\\
\end{equation}

\pagebreak
\begin{ex} We picture the inclusion of root systems $C_1\subset C_2$
and the map of fans $\Sigma(C_2)\to\Sigma(C_1)$.

\smallskip

\noindent
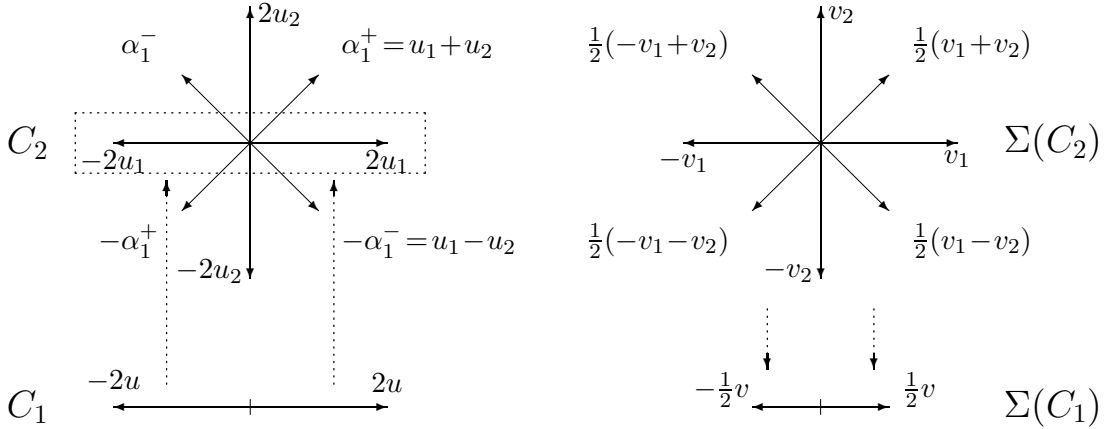
\begin{picture}(150,60)(0,0)

\put(3,40){\makebox(0,0)[l]{\large$C_2$}}
\put(3,5){\makebox(0,0)[l]{\large$C_1$}}
\put(35,40){\vector(0,1){18}}\put(35,40){\vector(0,-1){18}}
\put(35,40){\vector(1,0){18}}\put(35,40){\vector(-1,0){18}}
\put(35,40){\vector(1,1){9}}\put(35,40){\vector(-1,1){9}}
\put(35,40){\vector(1,-1){9}}\put(35,40){\vector(-1,-1){9}}
\dottedline{1}(12,44)(58,44)\dottedline{1}(12,36)(58,36)
\dottedline{1}(12,44)(12,36)\dottedline{1}(58,44)(58,36)
\put(53,36){\makebox(0,0)[b]{\small$2u_1$}}
\put(17,36){\makebox(0,0)[b]{\small$-2u_1$}}
\put(47,53){\makebox(0,0)[l]{\small$\alpha_1^+\!=\!u_1\!+\!u_2$}}
\put(36,57){\makebox(0,0)[l]{\small$2u_2$}}
\put(23,53){\makebox(0,0)[r]{\small$\alpha_1^-$}}
\put(47,27){\makebox(0,0)[l]{\small$-\alpha_1^-\!=\!u_1\!-\!u_2$}}
\put(34,23){\makebox(0,0)[r]{\small$-2u_2$}}
\put(23,27){\makebox(0,0)[r]{\small$-\alpha_1^+$}}
\put(24,33){\vector(0,1){2}}\put(46,33){\vector(0,1){2}}
\dottedline{1}(24,8)(24,35)\dottedline{1}(46,8)(46,35)
\put(35,5){\vector(1,0){18}}\put(35,5){\vector(-1,0){18}}
\put(53,7){\makebox(0,0)[b]{\small$2u$}}
\put(17,7){\makebox(0,0)[b]{\small$-2u$}}
\drawline(35,4)(35,6)

\put(75,0){
\put(72,40){\makebox(0,0)[r]{\large$\Sigma(C_2)$}}
\put(72,5){\makebox(0,0)[r]{\large$\Sigma(C_1)$}}
\put(35,40){\vector(0,1){18}}\put(35,40){\vector(0,-1){18}}
\put(35,40){\vector(1,0){18}}\put(35,40){\vector(-1,0){18}}
\put(35,40){\vector(1,1){9}}\put(35,40){\vector(-1,1){9}}
\put(35,40){\vector(1,-1){9}}\put(35,40){\vector(-1,-1){9}}
\put(53,39){\makebox(0,0)[t]{\small$v_1$}}
\put(17,39){\makebox(0,0)[t]{\small$-v_1$}}
\put(47,53){\makebox(0,0)[l]{\small$\frac{1}{2}(v_1\!+\!v_2)$}}
\put(36,57){\makebox(0,0)[l]{\small$v_2$}}
\put(23,53){\makebox(0,0)[r]{\small$\frac{1}{2}(-v_1\!+\!v_2)$}}
\put(47,27){\makebox(0,0)[l]{\small$\frac{1}{2}(v_1\!-\!v_2)$}}
\put(34,23){\makebox(0,0)[r]{\small$-v_2$}}
\put(23,27){\makebox(0,0)[r]{\small$\frac{1}{2}(-v_1\!-\!v_2)$}}
\dottedline{1}(28,18)(28,12)\put(28,12){\vector(0,-1){2}}
\dottedline{1}(42,18)(42,12)\put(42,12){\vector(0,-1){2}}
\put(35,5){\vector(1,0){9}}\put(35,5){\vector(-1,0){9}}
\put(48,5){\makebox(0,0)[b]{\small$\frac{1}{2}v$}}
\put(22,5){\makebox(0,0)[b]{\small$-\frac{1}{2}v$}}
\drawline(35,4)(35,6)
}
\end{picture}
\end{ex}

The fibres of $X(C_{n+1})\to X(C_n)$ can be studied for example
using the above description in terms of equations or by employing
the description of $X(C_n)$ as quotient of $X(B_n)$ (see below).
We obtain the following result, in particular the fibres over a
union of torus invariant divisors are not reduced.

\begin{prop}\label{prop:combtypeC}
We define $D\subset X(C_n)$ to be the union of the torus invariant
divisors corresponding to the one-dimensional cones of $\Sigma(C_n)$
generated by elements of the form $\frac{1}{2}(\varepsilon_1v_1+\ldots
+\varepsilon_nv_n)$.
For the structure of the fibres of the morphism $X(C_{n+1})\to X(C_n)$
together with the involution $I$, the sections $s_i^\pm$ and the
double-section $S_0$, there are the following two situations.

Over $X(C_n)\setminus D$ the fibres are $B_n$-curves except that instead
of the section $s_0$ we have a double-section $S_0$ which consists of
the two fixed points under $I$. In this case the central component
contains some of the sections $s_i^\pm$.

Over $D$ the fibres are $B_n$-curves except that the central component
is nonreduced of the form $\P^1_{K[\varepsilon]/\langle\varepsilon^2
\rangle}$ with the double-section $S_0\cong\Spec K[\varepsilon]/
\langle\varepsilon^2\rangle$ concentrated in one point.
The intersection of the central component with the other components
locally is isomorphic to the subscheme in $\A^2_K=\Spec K[x,y]$
defined by the equation $x^2y=0$.
All sections $s_i^\pm$ are on the other components.

In both cases the combinatorial types over the torus invariant divisors,
after the appropriate modifications, are given by the description
in the $B$-case (prop. \ref{prop:combtypeBn}).
\end{prop}

\medskip
\noindent
{\bf\boldmath $X(C_n)$ as quotient of $X(B_n)$.}
We investigate the description of $X(C_n)$ as a quotient $X(B_n)/\mu_2$.
On the moduli side this leads to a characterisation of $X(C_n)$ as the
coarse moduli space of a toric Deligne-Mumford stack. For simplicity,
in this part we will work over the field of complex numbers.

\medskip

On the moduli space $\overline{L}_n^{0,\pm}$ of $B_n$-curves we have
an involution $J$ that transforms a $B_n$-curve over a scheme $Y$ to
the $B_n$-curve with the other fixed point section with respect to the 
involution $I$ as section $s_0$, i.e.\ we apply the automorphism of 
the cen-\linebreak tral component that commutes with $I$ (see the 
following remark) to the section $s_0$.

\begin{rem}\label{rem:autom}
Let $(C,I,s_-,s_+)$ be a chain of projective lines with involution of
odd length over $\C$.
Consider the central component $(C_0,p^-_0,p^+_0)$ which we identify
with $(\P^1_\C,0,\infty)$ such that $I|_{C_0}\colon x\mapsto\frac{1}{x}$.
Then there are two automorphisms of $(C_0,p^-_0,p^+_0)$ that commute
with $I$, namely the identity and $x\mapsto-x$, determined by the
action on the fixed points $\{1,-1\}$ of $I|_{C_0}$.
\end{rem}

Identifying $\overline{L}_n^{0,\pm}$ with $X(B_n)$, the involution
$J$ is given on the functor of $B_n$-data (see \cite[1.3]{BB11}) by
$(\mathscr L_{\{\pm u_i\}},\{t_{u_i},t_{-u_i}\})\mapsto
(\mathscr L_{\{\pm u_i\}},\{t_{u_i},-t_{-u_i}\})$ or equivalently
$f_{\pm u_i}\mapsto -f_{\pm u_i}$ on the part corresponding to the
roots $\pm u_1,\ldots,\pm u_n$, whereas the part corresponding to
the other roots remains unchanged.

\medskip

In both the $C_n$-case and the $B_n$-case we start with the same
vector space $E$ with basis $u_1,\ldots,u_n$. The root lattice
$M(C_n)$ is a sublattice of the root lattice $M(B_n)$ of index $2$
and dually $N(B_n)\subset N(C_n)$ of index $2$, whereas the
fan $\Sigma(C_n)$ as a set of cones in $N(C_n)_\Q=N(B_n)_\Q$ is the
same as the fan $\Sigma(B_n)$. Thus, the toric variety $X(C_n)$ is
the quotient of $X(B_n)$ by the involution that maps
$x^{u_i}\mapsto -x^{u_i}$. This involution on $X(B_n)$ coincides
with the involution $J$. Locally, we have quotients
$\A^n/\mu_2$ by the action of $\mu_2$ that changes the sign
of one coordinate of $\A^n$. In particular, $X(B_n)$ is flat over
$X(C_n)$ of degree $2$. $X(C_n)$ can be considered as the
$\mu_2$-Hilbert scheme of $X(B_n)$, then $X(B_n)\to X(C_n)$ forms
the universal family of $\mu_2$-clusters, the fibres over
$X(C_n)\setminus D$ consist of two points, the fibres over $D$
are nonreduced $\mu_2$-clusters.

\medskip

Concerning the double-section $S_0\subset X(C_{n+1})$ we obtain:

\begin{lemma}
The scheme $S_0$ is isomorphic to $X(B_n)$ over $X(C_n)$.
\end{lemma}
\begin{proof}
Let $\tilde{S}_0\subset X(B_{n+1})$ be the fixed point subscheme
of the involution $I$ on $X(B_{n+1})$.
The scheme $\tilde{S}_0$ over $X(B_n)$ consists of two components
$s_0(X(B_n))$ and another copy of $X(B_n)$ such that
$J\colon X(B_{n+1})\to X(B_{n+1})$ restricts to an isomorphism
between these components over $J\colon X(B_n)\to X(B_n)$.
The scheme $S_0$ arises as quotient of $\tilde{S}_0$ by $J$,
the section $s_0\colon X(B_n)\to\tilde{S}_0$ determines an
isomorphism $X(B_n)\to S_0$ over $X(C_n)=X(B_n)/\mu_2$.
\end{proof}

We are led to the following type of curves to be parametrised
by $X(C_n)$.

\begin{defi} (First definition of $C_n$-curves).
A $C_n$-curve over a scheme $Y$ is a collection
$(\pi\colon C\to Y,I,s_-,s_+,$ $s_1^\pm,\ldots,s_n^\pm)$ which
arises from a $B_n$-curve over $Y$ by omitting the section $s_0$.
\end{defi}

Equivalently, we can replace the section $s_0$ of a $B_n$-curve
$C\to Y$ by the subscheme $s_0(Y)\cup J(s_0(Y))$, which coincides
with the fixed point subscheme of the involution $I$ on $C$.
The section $s_0$ selects one of the two components of this fixed
point subscheme. Forgetting this information, the $B_n$-curves for
points $y,Jy$ in the moduli space $\overline{L}_n^{0,\pm}\!=\!X(B_n)$
define $C_n$-curves related by an isomorphism of $C_n$-curves.
If the central component contains sections $s_i^\pm$, then two
nonisomorphic $B_n$-curves over a field give rise to isomorphic
$C_n$-curves.
If the central component does not contain a section $s_i^\pm$, then
one $B_n$-curve corresponds to one $C_n$-curve, but $C_n$-curves
of this type have an extra automorphism that interchanges the two
fixed points of $I$ (cf.\ remark \ref{rem:autom}).

\pagebreak

This functor of $C_n$-curves cannot be representable by a scheme.
However, we can consider the stack of $C_n$-curves.

\begin{thm}
The category of $C_n$-curves forms a Deligne-Mumford stack
$\mathcal X(C_n)$ isomorphic to the quotient stack $[X(B_n)/\mu_2]$
with the group operation given by $J\colon X(B_n)\to X(B_n)$.
\end{thm}

\begin{proof}
Let $\mathcal X(C_n)$ be the category of $C_n$-curves, i.e.\
an object of $\mathcal X(C_n)$ over a scheme $Y$ is a $C_n$-curve
$C$ over $Y$, a morphism $(C\to Y)\to(C'\to Y')$ over $Y\to Y'$
is a cartesian diagram compatible with the involution $I$ and
the sections. This is a category fibred in groupoids, we show
that it is equivalent as a fibred category to the Deligne-Mumford
stack $[X(B_n)/\mu_2]$.

An object of $[X(B_n)/\mu_2]$ over a scheme $Y$ is a
$\mu_2$-torsor $T\to Y$ together with a $\mu_2$-equivariant
map $T\to X(B_n)$. A morphism $(T\to Y,\alpha\colon T\to X(B_n))
\to(T'\to Y',\alpha'\colon T'\to X(B_n))$ over $Y\to Y'$ is a
cartesian diagram of $\mu_2$-torsors given by a morphism
$\theta\colon T\to T'$ such that $\alpha'\circ\theta=\alpha$.
We will use that the functor of $X(B_n)$ is isomorphic to the
functor of $B_n$-curves and fix an isomorphism resp.\ a universal
family over $X(B_n)$.

We define a morphism of fibred categories $\Phi\colon[X(B_n)/\mu_2]
\to\mathcal X(C_n)$. For an object $(T\to Y,\alpha\colon T\to X(B_n))$
we have a $B_n$-curve $B\to T$ corresponding to the equivariant morphism
$\alpha$ such that the action of $\mu_2$ on $T$ is given by
interchanging the two possible choices of $s_0$.
After forgetting the section $s_0$, the quotient of $B\to T$ by
$\mu_2$ gives a $C_n$-curve $C\to Y$ using the canonical isomorphism
$T/\mu_2\cong Y$.
For a morphism $(T\to Y,\alpha\colon T\to X(B_n))\to
(T'\to Y,\alpha'\colon T'\to X(B_n))$ we obtain a cartesian
diagram of $C_n$-curves $(C\to Y)\to (C'\to Y')$.

We define a morphism of fibred categories $\Psi\colon\mathcal X(C_n)
\to[X(B_n)/\mu_2]$. Let $C\to Y$ be a $C_n$-curve over $Y$.
Consider the fixed point subscheme $T\subset C$ under $I$, this is
a $\mu_2$-torsor over $Y$. Let $B$ be the pull-back of the $C_n$-curve
$C\to Y$ to $T$, with the section $s_0$ defined as the diagonal of
$T\times_Y T\subset B$ this is a $B_n$-curve and defines a
$\mu_2$-equivariant morphism $\alpha\colon T\to X(B_n)$.
A morphism $(C\to Y)\to(C'\to Y')$ given by $\gamma\colon C\to C'$
determines a cartesian diagram of $B_n$-curves by
$\gamma\times\gamma\colon B=C\times_Y T\to B'=C'\times_{Y'}T'$
over a cartesian diagram of $\mu_2$-torsors given by
$\gamma\colon T\to T'$. So the diagram formed by
$\gamma\colon T\to T'$ and $T,T'\to X(B_n)$ is commutative.

The compositions $\Phi\circ\Psi$ and $\Psi\circ\Phi$ are isomorphic
to the respective identities.
In the case of $\Phi\circ\Psi$ the quotient of the pull-back of a
$C_n$-curve $C\to Y$ to $T\subset C$ is canonically isomorphic to
$C\to Y$.
In the case of $\Psi\circ\Phi$ the quotient of a $B_n$-curve
$B\to T$ over a $\mu_2$-torsor $T$ gives a $C_n$-curve $C\to Y$,
together these form a cartesian square. The section
$s_0\colon T\to B$ determines an inclusion $T\subset C$ as fixed
point subscheme with respect to $I$. Applying the functor $\Psi$
we recover a $B_n$-curve canonically isomorphic to the original
$B_n$-curve.
\end{proof}

\begin{cor}
The toric variety $X(C_n)$ is a coarse moduli space of $C_n$-curves.
\end{cor}

\medskip

The stack $\mathcal X(C_n)$ is a toric Deligne-Mumford stack
as introduced in \cite{BCS04} (see also \cite{FMN10}):
we define the stacky fan $\mathbf\Sigma(C_n)$ as the fan $\Sigma(C_n)$
in the lattice $N(C_n)$ with the difference that we choose on the rays
generated by $\frac{1}{2}(\varepsilon_1v_1+\ldots+\varepsilon_nv_n)$
the second lattice points $\varepsilon_1v_1+\ldots+\varepsilon_nv_n$.
In comparision to the fan $\Sigma(B_n)$ the underlying lattice is
finer and the toric DM stack associated with $\bold\Sigma(C_n)$
coincides with the quotient stack $[X(B_n)/\mu_2]$.

\begin{cor}
The stack $\mathcal X(C_n)$ is isomorphic to the toric Deligne-Mumford
stack associated with the stacky fan $\bold\Sigma(C_n)$.
\end{cor}

\begin{ex}
The stacky fan $\bold\Sigma(C_2)$ in the lattice $\Z\frac{1}{2}v\cong\Z$
consists of the two cones $\Q_{\geq0}v,\Q_{\geq0}(-v)$ with
chosen lattice points $v,-v$. The associated toric DM stack is
$\mathcal X(C_2)\cong[\P^1/\mu_2]$ (cf.\ also \cite[example 7.31]{FMN10}),
it is an orbifold with two stacky points.
\end{ex}

\medskip
\noindent
{\bf\boldmath $X(C_n)$ as fine moduli space.}
We give a characterisation of $X(C_n)$ as a fine moduli space
$\overline{L}_n^\pm$ of $2n$-pointed chains of projective lines.
Here the universal curve is not $X(C_{n+1})\to X(C_n)$, however,
the universal curve and the general notion of a $C_n$-curve are
defined naturally in terms of the inclusion of root systems
$C_n\to C_{n+1}$.

\medskip

We have the root subsystem $C_n\subset C_{n+1}$ in the subspace
generated by the roots $u_1,\ldots,u_n$.
Take those pairs of opposite roots in $C_{n+1}\setminus C_n$ which define
projections $C_{n+1}\to C_n$ in the sense of \cite[1.2]{BB11}; these
are $\pm\alpha_1^-,\pm\alpha_1^+,\ldots,\pm\alpha_n^-,\pm\alpha_n^+$
but not $\pm2u_{n+1}$.
To each of these pairs $\pm\alpha_i^-$ and $\pm\alpha_i^+$ we associate
a section $s_i^-$ and $s_i^+$.
The element of the Weyl group given as the reflection for the root
$\pm2u_{n+1}$ mapping $u_{n+1}\mapsto-u_{n+1}$ and $u_i\mapsto u_i$
for $i\in\{1,\ldots,n\}$ is an isomorphism of $C_{n+1}$ fixing
$C_n\subset C_{n+1}$. It maps $\alpha_i^-\leftrightarrow-\alpha_i^+$.
This leads us to the following definition.

\begin{defi} (Second definition of $C_n$-curves).
A $C_n$-curve over an algebraically closed field $K$ is a chain of
projective lines with involution of odd or even length with $2n$
(possibly coinciding) marked points $s_1^\pm,\ldots,s_n^\pm$ different
from the poles, the involution interchanging $s_i^-\leftrightarrow s_i^+$,
such that every component contains at least one of the points $s_i^\pm$.
We define a $C_n$-curve over an arbitrary scheme, isomorphisms of
$C_n$-curves and the moduli functor of $C_n$-curves in the same way
as we did in the case of $B_n$-curves.
\end{defi}

\begin{con}
Let the subscheme
\[\textstyle
C(C_{n+1}/C_n)\subset
\big(\prod_{i=1}^n\P^1_{\{\pm\alpha_i^-\}}\times
\prod_{i=1}^n\P^1_{\{\pm\alpha_i^+\}}\big)_{X(C_n)}
\]
be defined by the equations (\ref{eq:CnII}) using the universal
$C_n$-data on $X(C_n)$. This morphism $C(C_{n+1}/C_n)\to X(C_n)$
has the sections $s_-,s_+,s_i^\pm$, where $s_-$ is defined by
$z_{-\alpha_i^\pm}=0$ ($i=1,\ldots,n$), $s_+$ is defined by
$z_{\alpha_i^\pm}=0$ ($i=1,\ldots,n$) and the sections $s_i^\pm$
by the equations $z_{\alpha_i^\pm}=z_{-\alpha_i^\pm}$.
The involution maps $\P^1_{\{\pm\alpha_i^-\}}\leftrightarrow
\P^1_{\{\pm\alpha_i^+\}}$, $(z_{\alpha_i^-}:z_{-\alpha_i^-})
\leftrightarrow(z_{-\alpha_i^+}:z_{\alpha_i^+})$.
\end{con}

\begin{rem}
The toric variety $C(C_{n+1}/C_n)$ arises from $X(C_{n+1})$ by
contracting certain torus invariant prime divisors. The fibres of
$X(C_{n+1})\to X(C_n)$ over the divisors corresponding 
to the rays generated by elements of the form
$\frac{1}{2}(\varepsilon_1v_1+\ldots+\varepsilon_nv_n)$ (forming $D$
in proposition \ref{prop:combtypeC}) have a central component
containing none of the sections $s_i^\pm$. In $X(C_{n+1})$ the
support of the central components of the fibers over the divisor
corresponding to
$\frac{1}{2}(\varepsilon_1v_1+\ldots+\varepsilon_nv_n)$ forms a
torus invariant divisor which corresponds to the ray in
$\Sigma(C_{n+1})$ generated by $\varepsilon_1v_1+\ldots
+\varepsilon_nv_n$ and is isomorphic to $X(C_1)\times X(A_{n-1})
\cong\P^1\times X(A_{n-1})$. We contract these divisors\linebreak
$\P^1\times X(A_{n-1})$ to $X(A_{n-1})$ by omitting the rays in
$\Sigma(C_{n+1})$ generated by \linebreak elements of the form
$\varepsilon_1v_1+\ldots+\varepsilon_nv_n$, but retaining the
two-dimensional cones \linebreak 
$\left<
\frac{1}{2}(\varepsilon_1v_1+\ldots+\varepsilon_nv_n-v_{n+1}),
\frac{1}{2}(\varepsilon_1v_1+\ldots+\varepsilon_nv_n+v_{n+1})
\right>_{\Q_{\geq0}}$. 
On the fibers over $D$ the central components are contracted.
\end{rem}

\begin{prop}
The morphism $C(C_{n+1}/C_n)\to X(C_n)$ with the involution $I$ and
the sections $s_-,s_+,s_1^\pm,\ldots,s_n^\pm$ is a $C_n$-curve. The
combinatorial types of the fibres over the torus orbits corresponding
to one-dimensional cones are as follows:
\[
\begin{array}{lc}
\varepsilon_{i_1}v_{i_1}&s_{i_1}^{\varepsilon_1}|s_{i_2}^\pm
\cdots s_{i_n}^\pm|s_{i_1}^{-\varepsilon_1}\\

\varepsilon_{i_1}v_{i_1}+\varepsilon_{i_2}v_{i_2}&\;\;
s_{i_1}^{\varepsilon_1}s_{i_2}^{\varepsilon_2}|s_{i_3}^\pm\cdots
s_{i_n}^\pm|s_{i_2}^{-\varepsilon_2}s_{i_1}^{-\varepsilon_1}\\
\vdots&\vdots\\

\varepsilon_{i_1}v_{i_1}+\;\ldots\;+\varepsilon_{i_{n-2}}v_{i_{n-2}}&\;\;\;
s_{i_1}^{\varepsilon_1}\cdots s_{i_{n-2}}^{\varepsilon_{n-2}}|
s_{i_{n-1}}^\pm s_{i_n}^\pm|s_{i_{n-2}}^{-\varepsilon_{n-2}}\cdots
s_{i_1}^{-\varepsilon_1}\\

\varepsilon_{i_1}v_{i_1}+\;\ldots\ldots\;+\varepsilon_{i_{n-1}}v_{i_{n-1}}
\hspace{-1cm}&\;\;\;\;
s_{i_1}^{\varepsilon_1}\cdots s_{i_{n-1}}^{\varepsilon_{n-1}}|
s_{i_{n}}^\pm|s_{i_{n-1}}^{-\varepsilon_{n-1}}
\cdots s_{i_1}^{-\varepsilon_1}\\

\frac{1}{2}(\varepsilon_{i_1}v_{i_1}+\;\ldots\ldots\ldots\;
+\varepsilon_{i_{n}}v_{i_{n}})&\;\;\;\;
s_{i_1}^{\varepsilon_1}\cdots s_{i_{n}}^{\varepsilon_{n}}|
s_{i_{n}}^{-\varepsilon_{n}}\cdots s_{i_1}^{-\varepsilon_1}\\
\end{array}
\]
\end{prop}

\begin{defi}
We call $C(C_{n+1}/C_n)\to X(C_n)$ together with the involution $I$
and the sections $s_-,s_+,s_1^-,s_1^+,\ldots,s_n^-,s_n^+$ the universal
$C_n$-curve over $X(C_n)$.
\end{defi}

By the same procedure as in the case of root systems of type $A$
and $B$ we can prove the following.

\begin{thm}
There exists a fine moduli space $\overline{L}_n^\pm$ of $C_n$-curves
isomorphic to the toric variety $X(C_n)$ with universal family
$C(C_{n+1}/C_n)\to X(C_n)$.
\end{thm}

\begin{rem}
There is a natural closed embedding of the moduli spaces
$\overline{L}_n^\pm\!=\!X(C_n)\,\to\,\overline{L}_{2n}\!=\!X(A_{2n-1})$
determined by considering a $C_n$-curve with sections
$s_1^-,\ldots,s_n^-,s_n^+,\ldots,s_1^+$ as an $A_{2n-1}$-curve with
sections $s_1,\ldots,s_{2n}$. The toric morphism $X(C_n)\to X(A_{2n-1})$
is given by the projection of root systems $A_{2n-1}\to C_n$ induced
by $\bigoplus_{i=1}^{2n}\Z u_i\to M(C_n)$, $u_i\mapsto u_i$,
$u_{2n+1-i}\mapsto-u_i$ for $i=1,\ldots,n$. The kernel in $M(A_{2n-1})$
is generated by $u_{2n+1-i}+u_i-u_{2n+1-j}-u_j$ for some
fixed $j$ and $i\in\{1,\ldots,n\}\setminus\{j\}$.
By employing this embedding we have an alternative approach to prove
the above statements.
\end{rem}

\bigskip
\section{Root systems of type $D$}
\label{sec:D_n}

Consider for $n\geq 2$ an $n$-dimensional Euclidean space $E$ with
basis $u_1,\ldots,u_n$. The root system $D_n$ in $E$ consists of the
$2n(n-1)$ roots
\[
\pm(u_i+u_j), \pm(u_i-u_j)\;\:\textit{for}\;\:i,j\in\{1,\ldots,n\},i<j.
\]
The following is a set of simple roots:
\[u_1-u_2,u_2-u_3,\ldots,u_{n-1}-u_n,u_{n-1}+u_n.\]
The Weyl group $(\Z/2\Z)^{n-1}\rtimes S_n$ acts by $u_i\mapsto
\varepsilon_i u_i$, where the $\varepsilon_i$ are signs such that
$\prod_i\varepsilon_i=1$, and by permuting the $u_i$.
So there are $2^{n-1}n!$ sets of simple roots, these are of the form
$\varepsilon_1u_{i_1}-\varepsilon_2u_{i_2},
\varepsilon_2u_{i_2}-\varepsilon_3u_{i_3},\ldots,
\varepsilon_{n-1}u_{i_{n-1}}-\varepsilon_{n}u_{i_{n}},
\varepsilon_{n-1}u_{i_{n-1}}+\varepsilon_{n}u_{i_{n}}$
for orderings $i_1,\ldots,i_n$ of the set
$\{1,\ldots,n\}$ and signs $\varepsilon_1,\ldots,\varepsilon_n$
(note that $\varepsilon_n=1$ and $\varepsilon_n=-1$ give the same set).

\medskip

The vector space $E^*$ dual to $E$ with basis $v_1,\ldots,v_n$
dual to $u_1,\ldots,u_n$ contains the lattice $N(D_n)$ dual to the
root lattice $M(D_n)$. To describe the fan $\Sigma(D_n)$ in the
lattice $N(D_n)$ we determine a Weyl chamber. The set of simple
roots $u_1-u_2,u_2-u_3,\ldots,u_{n-1}-u_n,u_{n-1}+u_n$ has the
dual basis $v_1,v_1+v_2,\ldots,v_1+\ldots+v_{n-2},\frac{1}{2}(v_1+
\ldots+v_{n-1}-v_n),\frac{1}{2}(v_1+\ldots+v_{n-1}+v_n)$ of $N(D_n)$
which generates the corresponding Weyl chamber.
There are $3^n-n2^{n-1}-1$ one-dimensional cones generated by
elements of the form $\sum_{i\in A}\varepsilon_iv_i$ for
$A\subset\{1,\ldots,n\}$, $1\leq|A|\leq n-2$ or of the form
$\frac{1}{2}(\varepsilon_1v_1+\ldots+\varepsilon_{n-1}v_{n-1}+
\varepsilon_nv_n)$, where the $\varepsilon_i$ are signs.

\medskip

The torus invariant divisor for the one-dimensional cone generated by
$\varepsilon_1v_{i_1}+\ldots+\varepsilon_kv_{i_k}$, $1\leq k\leq n-2$
is isomorphic to $X(D_{n-k})\times X(A_{k-1})$, that for
$\varepsilon_1v_1+\ldots+\varepsilon_{n-2}v_{n-2}$ is isomorphic to
$X(A_1)\times X(A_1)\times X(A_{n-3})\cong X(D_2)\times X(A_{n-3})$
and that for $\frac{1}{2}(\varepsilon_1v_1+\ldots+\varepsilon_nv_n)$
is isomorphic to $X(A_{n-1})$ (see \cite[1.2]{BB11}).

\bigskip
\noindent
{\bf\boldmath $X(D_{n+1})$ over $X(D_n)$.}
Consider the proper surjective morphism $X(D_{n+1})\to X(D_n)$ induced
by the root subsystem $D_n\subset D_{n+1}$ consisting of the roots
in the subspace generated by $u_1,\ldots,u_n$.
We have a projection of fans $\Sigma(D_{n+1})\to \Sigma(D_n)$ along
the subspace generated by $v_{n+1}$. The generic fibre is $\P^1$.
Note that the torus invariant divisor in $X(D_{n+1})$ corresponding to
$v_1+\ldots+v_{n-1}$ is lying over the closure of the torus orbit in
$X(D_n)$ of codimension $2$ corresponding to the 2-dimensional cone
generated by $\frac{1}{2}(v_1+\ldots+v_{n-1}+v_n),
\frac{1}{2}(v_1+\ldots+v_{n-1}-v_n)$; here (and on the translates under
the Weyl group $W(B_n)$) we have fibres of dimension $2$. This implies
that the morphism $X(D_{n+1})\to X(D_n)$ is not flat.

\medskip

There are $2n$ additional pairs of opposite roots, the pairs
$\pm\alpha_i^+=\pm(u_{n+1}+u_i)$ and $\pm\alpha_i^-=\pm(u_{n+1}-u_i)$
for $i\in\{1,\ldots,n\}$.
The projections along the subspaces generated by these do not
define projections of root systems $D_{n+1}\to D_n$ in the sense
of \cite[1.2]{BB11}: we have $\alpha_i^+-\alpha_i^-=2u_i$, so the
projection along the subspace generated by $\alpha_i^+$
(resp.\ $\alpha_i^-$) maps $\alpha_i^-$ (resp.\ $\alpha_i^+$)
to $2u_i$ which is not a multiple of a root of $D_n$.
Instead we can consider the preimages of
$(1:1)\in\P^1_{\{\pm\alpha_i^-\}},\P^1_{\{\pm\alpha_i^+\}}$
with respect to the projections $X(D_{n+1})\to\P^1_{\{\pm\alpha_i^-\}},
\P^1_{\{\pm\alpha_i^+\}}$ determined by the inclusions of root systems
$\{\pm\alpha_i^-\},\{\pm\alpha_i^+\}\subset D_{n+1}$, we denote these
subschemes by $s_i^-,s_i^+$. As in the $B$ and $C$-case we have
sections $s_-,s_+$; further we have an involution $I$ coming from the
automorphism of $D_{n+1}$ fixing $D_n\subset D_{n+1}$ which maps
$u_{n+1}\mapsto -u_{n+1}$, $u_i\mapsto u_i$ for $i\in\{1,\ldots,n\}$
and is not an element of the Weyl group $W(D_{n+1})$.

\medskip

As in the other cases we can study $X(D_{n+1})$ over $X(D_n)$
via the embedding into
$P(D_{n+1}/D_n)_{X(D_n)}=\big(\prod_{i=1}^n\P^1_{\{\pm\alpha_i^-\}}
\times\prod_{i=1}^n\P^1_{\{\pm\alpha_i^+\}}\big)_{X(D_n)}$.
The subscheme $X(D_{n+1})$ $\subset P(D_{n+1}/D_n)_{X(D_n)}$ is given
by the homogeneous equations parametrised by the universal $D_n$-data
\[
t_{\beta}z_{\alpha_2}z_{-\alpha_1}=t_{-\beta}z_{-\alpha_2}z_{\alpha_1}
\qquad
\begin{array}{l}
\textit{for}\;\alpha_1,\alpha_2\in \{\alpha_1^\pm,\ldots,\alpha_n^\pm\}\\
\textit{such that}\;\beta=\alpha_1-\alpha_2\;\textit{is a root of $D_n$}\\
\end{array}
\]

\pagebreak
We will see that over the complement of a closed subset of
codimension $2$ the fibres are chains of projective lines
with sections $s_i^\pm$. Over these points we have a combinatorial
type for a fibre resp.\ for the universal $D_n$-data as in the
$B$-case (see proposition \ref{prop:combtypeBn}), we use the
notation introduced there.

\begin{ex} $X(D_3)$ over $X(D_2)$.\\
The root system $D_2$ consists of the $4$ roots $\pm u_1\pm u_2$.
It is contained in the root system $D_3$, this has the $8$ additional
roots $\pm\alpha_1^-=\pm(u_3-u_1)$, $\pm\alpha_1^+=\pm(u_3+u_1)$,
$\pm\alpha_2^-=\pm(u_3-u_2)$, $\pm\alpha_2^+=\pm(u_3+u_2)$.
Because of the isomorphism of root systems $D_2\cong A_1\times A_1$
we have $X(D_2)\cong\P^1\times\P^1$. The fan $\Sigma(D_2)$ has $4$
one-dimensional cones generated by $\frac{1}{2}(\pm v_1\pm v_2)$.
The fan $\Sigma(D_3)$ has $14$ one-dimensional cones, $6$ of the
form $\pm v_i$ and $8$ of the form $\frac{1}{2}(\varepsilon_1v_1+
\varepsilon_2v_2+\varepsilon_3v_3)$.
The projection $\Sigma(D_3)\to\Sigma(D_2)$ maps the generator of the
one-dimensional cone
$\frac{1}{2}(\varepsilon_1v_1+\varepsilon_2v_2+\varepsilon_3v_3)$
to the generator of the one-dimensional cone
$\frac{1}{2}(\varepsilon_1v_1+\varepsilon_2v_2)$,
the vector $\pm v_i$ for $i=1,2$ is not mapped to a one-dimensional
cone of $D_2$ but into the interior of the $2$-dimensional cone
$\langle\pm v_i+v_j,\pm v_i-v_j\rangle_{\Q_{\geq 0}}$.\\
In $P(D_3/D_2)_{X(D_2)}=\big(\P^1_{\{\pm\alpha_1^-\}}\times
\P^1_{\{\pm\alpha_1^+\}}\times\P^1_{\{\pm\alpha_2^-\}}\times
\P^1_{\{\pm\alpha_2^+\}}\big)_{X(D_2)}$ the subscheme $X(D_3)$
is given by $4$ equations parametrised by the universal $D_2$-data
on $X(D_2)$. For each point we have $D_2$-data of the form
$(t_{\beta_{12}}:t_{-\beta_{12}})$, $(t_{\gamma_{12}}:t_{-\gamma_{12}})$
where $\beta_{12}=u_1-u_2$, $\gamma_{12}=u_1+u_2$.
Over the affine chart $\Spec\Z[\frac{x_1}{x_2},x_1x_2]$
corresponding to the cone $\langle\frac{1}{2}(v_1-v_2),
\frac{1}{2}(v_1+v_2)\rangle_{\Q_{\geq 0}}$ for the set of
simple roots $\beta_{12},\gamma_{12}$ this data has the property
$(t_{\beta_{12}}:t_{-\beta_{12}})\neq(1:0)$,
$(t_{\gamma_{12}}:t_{-\gamma_{12}})\neq(1:0)$
(see \cite[Rem.\ 1.21]{BB11}).
We study the fibres of $X(D_3)\to X(D_2)$ over this affine chart.
Over the dense torus we have a $\P^1$, over the torus orbit
corresponding to $\frac{1}{2}(v_1-v_2)$
(resp.\ $\frac{1}{2}(v_1+v_2)$)
we have chains of two $\P^1$ of combinatorial type
$s_1^+s_2^-|s_1^-s_2^+$ (resp.\ $s_1^+s_2^+|s_1^-s_2^-$).
Over the torus fixed point corresponding to the cone
$\langle\frac{1}{2}(v_1-v_2),\frac{1}{2}(v_1+v_2)\rangle_{\Q_{\geq 0}}$
we have $D_2$-data of the form
$(t_{\beta_{12}}:t_{-\beta_{12}})=(0:1)$,
$(t_{\gamma_{12}}:t_{-\gamma_{12}})=(0:1)$
and the fibre decomposes into three irreducible components
$\P^1,\P^1\times\P^1,\P^1$.

\noindent
\begin{picture}(150,45)(0,0)
\put(45,15){\line(1,0){40}}
\put(105,25){\line(-1,0){40}}
\put(55,20){\dottedline{1}(0,0)(40,0)}
\put(46,13){\line(2,1){28}}
\put(104,27){\line(-2,-1){28}}
\put(61,13){\dottedline{1}(0,0)(28,14)}
\put(50,20){\line(0,-1){20}}
\put(100,20){\line(0,1){20}}
\put(104,38){\makebox(0,0)[l]{$\P^1$}}
\put(104,33){\makebox(0,0)[l]{\Small$(z_{\alpha_1^-}:z_{-\alpha_1^-})$}}
\put(27,7){\makebox(0,0)[l]{$\P^1$}}
\put(27,2){\makebox(0,0)[l]{\Small$(z_{\alpha_1^+}:z_{-\alpha_1^+})$}}
\put(52,35){\makebox(0,0)[l]{$\P^1\times\P^1$}}
\put(52,30){\makebox(0,0)[l]{\Small$(z_{\alpha_2^-}:z_{-\alpha_2^-}),
(z_{\alpha_2^+}:z_{-\alpha_2^+})$}}
\put(62,13){\makebox(0,0)[t]{\small$s_2^-$}}
\put(96,20){\makebox(0,0)[t]{\small$s_2^+$}}
\filltype{white}
\put(50,2){\circle*{1.2}}
\put(51,1){\makebox(0,0)[l]{\small$s_-$}}
\put(100,38){\circle*{1.2}}
\put(99,37){\makebox(0,0)[r]{\small$s_+$}}
\filltype{black}
\put(50,8){\circle*{1.2}}
\put(51,8){\makebox(0,0)[l]{\small$s_1^+$}}
\put(100,32){\circle*{1.2}}
\put(99,32){\makebox(0,0)[r]{\small$s_1^-$}}
\end{picture}
\end{ex}

The general case can be studied using the same methods, see also
the $B_n$-case and in particular proposition \ref{prop:combtypeBn},
here details will be left to the reader. We \linebreak
define $Z\subset X(D_n)$
to be the union of the closures of torus orbits corresponding to
the $2$-dimensional cones of the form
$\langle\frac{1}{2}(\varepsilon_1v_{i_1}+\ldots+
\varepsilon_{n-1}v_{i_{n-1}}+\varepsilon_{i_n}v_{i_n}),\linebreak
\frac{1}{2}(\varepsilon_1v_{i_1}+\ldots+\varepsilon_{n-1}v_{i_{n-1}}
-\varepsilon_nv_{i_n})\rangle_{\Q_{\geq0}}$.

\pagebreak
\begin{prop}
Over $X(D_n)\setminus Z$ the fibres of the morphism $X(D_{n+1})\to
X(D_n)$ are chains of projective lines of odd or even length with
sections $s_i^\pm$. The combinatorial types of the fibres over the
torus orbits corresponding to one-dimensional cones are as follows:
\[
\begin{array}{lc}
\varepsilon_{i_1}v_{i_1}&s_{i_1}^{\varepsilon_1}|s_{i_2}^\pm
\cdots s_{i_n}^\pm|s_{i_1}^{-\varepsilon_1}\\
\varepsilon_{i_1}v_{i_1}+\varepsilon_{i_2}v_{i_2}&\;\;
s_{i_1}^{\varepsilon_1}s_{i_2}^{\varepsilon_2}|s_{i_3}^\pm\cdots
s_{i_n}^\pm|s_{i_2}^{-\varepsilon_2}s_{i_1}^{-\varepsilon_1}\\
\vdots&\vdots\\
\varepsilon_{i_1}v_{i_1}+\ldots+\varepsilon_{i_{n-2}}v_{i_{n-2}}&\;\;\;
s_{i_1}^{\varepsilon_1}\cdots s_{i_{n-2}}^{\varepsilon_{n-2}}|
s_{i_{n-1}}^\pm s_{i_n}^\pm|s_{i_{n-2}}^{-\varepsilon_{n-2}}\cdots
s_{i_1}^{-\varepsilon_1}\\
\frac{1}{2}(\varepsilon_{i_1}v_{i_1}+\ldots
+\varepsilon_{i_{n-2}}v_{i_{n-2}}+\varepsilon_{i_{n-1}}v_{i_{n-1}}
+\varepsilon_{i_{n}}v_{i_{n}})\hspace{-1cm}&\;\;\;\;
s_{i_1}^{\varepsilon_1}\cdots s_{i_{n}}^{\varepsilon_{n}}|
s_{i_{n}}^{-\varepsilon_{n}}\cdots s_{i_1}^{-\varepsilon_1}\\
\frac{1}{2}(\varepsilon_{i_1}v_{i_1}+\ldots
+\varepsilon_{i_{n-2}}v_{i_{n-2}}+\varepsilon_{i_{n-1}}v_{i_{n-1}}
-\varepsilon_{i_{n}}v_{i_{n}})\hspace{-1cm}&\;\;\;\;
s_{i_1}^{\varepsilon_1}\cdots s_{i_{n}}^{-\varepsilon_{n}}|
s_{i_{n}}^{\varepsilon_{n}}\cdots s_{i_1}^{-\varepsilon_1}\\
\end{array}
\]
Over $Z$ the fibres are $2$-dimensional and decompose into
irreducible components isomorphic to $\P^1$ and $\P^1\times\P^1$
intersecting transversally.
We have a central component $\P^1\times\P^1$ with action of $I$ that
interchanges two torus fixed points and leaves the other two fixed.
Further, we have chains of $\P^1$ emanating from the two torus fixed
points of $\P^1\times\P^1$ interchanged by $I$ with the sections
$s_\pm$ on the outer components. Concerning the subschemes $s_i^\pm$,
each of them intersects only with one component, those intersecting
with one of the $\P^1$ locally are sections, one pair $s_i^-,s_i^+$
intersects with $\P^1\times\P^1$ as $(1:1)\times\P^1$, $\P^1\times(1:1)$.
\end{prop}

\begin{rem}
The combinatorial type of fibres of $X(R_{n+1})$ over the torus fixed
points of $X(R_{n})$ can be pictured in form of the Dynkin
diagram of the root system $R_{n+1}$ such that a component $\P^1$
with one section corresponds to a vertex.\\
In the $A_n$-case (see \cite{BB11}) we have a string starting with the
section $s_{i_1}$ on the component containing $s_-$ and ending
with the section $s_{i_{n+1}}$ on the component containing $s_+$
in the form of the Dynkin diagram for the root system $A_{n+1}$:

\noindent
\begin{picture}(150,10)(0,0)
\put(20,4){\makebox(0,0)[c]{$s_{i_1}$}}
\put(25,4){\line(1,0){10}}
\put(40,4){\makebox(0,0)[c]{$s_{i_2}$}}
\put(45,4){\line(1,0){10}}
\put(60,4){\makebox(0,0)[c]{$s_{i_3}$}}
\put(65,4){\line(1,0){10}}
\put(90,4){\makebox(0,0)[c]{$.\quad.\quad.$}}
\put(105,4){\line(1,0){10}}
\put(121,4){\makebox(0,0)[c]{$s_{i_{n+1}}$}}
\end{picture}

\noindent
In the $B_n$-case, because of the involution $I$, it suffices to
consider the central component containing the section $s_0$ and
one of the two chains emanating from the central component. This
forms a Dynkin diagram of type $B_{n+1}$:

\noindent
\begin{picture}(150,10)(0,0)
\put(20,4){\makebox(0,0)[c]{$s_0$}}
\put(25,3.5){\line(1,0){10}}
\put(25,4.5){\line(1,0){10}}
\put(29,4){\line(1,1){2}}\put(29,4){\line(1,-1){2}}
\put(40,4){\makebox(0,0)[c]{$s_{i_1}^{\varepsilon_1}$}}
\put(45,4){\line(1,0){10}}
\put(60,4){\makebox(0,0)[c]{$s_{i_2}^{\varepsilon_2}$}}
\put(65,4){\line(1,0){10}}
\put(90,4){\makebox(0,0)[c]{$.\quad.\quad.$}}
\put(105,4){\line(1,0){10}}
\put(121,4){\makebox(0,0)[c]{$s_{i_n}^{\varepsilon_n}$}}
\end{picture}

\noindent
In the $C_n$-case we have the double-section $S_0$ replacing the
section $s_0$:

\noindent
\begin{picture}(150,10)(0,0)
\put(20,4){\makebox(0,0)[c]{$S_0$}}
\put(25,3.5){\line(1,0){10}}
\put(25,4.5){\line(1,0){10}}
\put(31,4){\line(-1,1){2}}\put(31,4){\line(-1,-1){2}}
\put(40,4){\makebox(0,0)[c]{$s_{i_1}^{\varepsilon_1}$}}
\put(45,4){\line(1,0){10}}
\put(60,4){\makebox(0,0)[c]{$s_{i_2}^{\varepsilon_2}$}}
\put(65,4){\line(1,0){10}}
\put(90,4){\makebox(0,0)[c]{$.\quad.\quad.$}}
\put(105,4){\line(1,0){10}}
\put(121,4){\makebox(0,0)[c]{$s_{i_n}^{\varepsilon_n}$}}
\end{picture}

\noindent
Finally, in the $D_n$-case we can take the torus invariant divisors
in the central component $\P^1\times\P^1$ and their intersection with
the fibres of the schemes $s_i^\pm$. Together with the other components
we have a Dynkin diagram of type $D_{n+1}$:

\noindent
\begin{picture}(150,20)(0,0)
\put(20,14){\makebox(0,0)[c]{$s_{i_1}^+$}}
\put(20,2){\makebox(0,0)[c]{$s_{i_1}^-$}}
\put(25,14){\line(2,-1){10}}
\put(25,2){\line(2,1){10}}
\put(40,8){\makebox(0,0)[c]{$s_{i_2}^{\varepsilon_2}$}}
\put(45,8){\line(1,0){10}}
\put(60,8){\makebox(0,0)[c]{$s_{i_3}^{\varepsilon_3}$}}
\put(65,8){\line(1,0){10}}
\put(90,8){\makebox(0,0)[c]{$.\quad.\quad.$}}
\put(105,8){\line(1,0){10}}
\put(121,8){\makebox(0,0)[c]{$s_{i_n}^{\varepsilon_n}$}}
\end{picture}
\end{rem}

\pagebreak

\bigskip

\end{document}